\pdfoutput=1
\documentclass[11pt]{amsart}

\usepackage[paper,pkg={amsaddr=true,tikz=true,biblatex/opt={+style=zbsnum}}]{zbs}
\numberwithin{equation}{section}
\allowdisplaybreaks

\newlist{cond}{enumerate}{2}
\setlist[cond,1]{label=(\arabic*),ref=\arabic*}
\setlist[cond,2]{label=(\arabic{condi}\alph*),ref=\arabic{condi}\alph*}
\crefname{condi}{condition}{conditions}
\Crefname{condi}{Condition}{Conditions}
\crefname{condii}{condition}{conditions}
\Crefname{condii}{Condition}{Conditions}
\crefname{scond}{condition}{conditions}
\Crefname{scond}{Condition}{Conditions}
\creflabelformat{enumi}{#2(#1)#3}
\creflabelformat{condi}{#2(#1)#3}
\creflabelformat{condii}{#2(#1)#3}
\creflabelformat{scond}{#2[#1]#3}

\newcommand{\va}{\bm{a}}
\newcommand{\vx}{\bm{x}}
\newcommand{\vz}{\bm{z}}
\newcommand{\vt}{\bm{t}}
\newcommand{\vY}{\mathbf{Y}}

\newcommand{\calh}{\mathcal{H}}
\newcommand{\caln}{\mathcal{N}}
\newcommand{\calp}{\mathcal{P}}
\newcommand{\cals}{\mathcal{S}}
\newcommand{\calm}{\mathcal{M}}
\newcommand{\calf}{\mathcal{F}}
\newcommand{\calr}{\mathcal{R}}
\newcommand{\calc}{\mathcal{C}}

\DeclareMathOperator{\aut}{aut}
\DeclareMathOperator{\Cov}{Cov}
\DeclareMathOperator{\diag}{diag}
\newcommand{\NIV}{\mathrm{NIV}~\mkern-5mu}
\newcommand{\HS}{\mathrm{HS}}
\newcommand{\op}{\mathrm{op}}

\makeatletter
\newcommand*\bigdot{{\mathpalette\bigdot@{.75}}}
\newcommand*\bigdot@[2]{\mathord{\vcenter{\hbox{\scalebox{#2}{$\m@th#1\mkern1mu\bullet\mkern1mu$}}}}}
\makeatother
\let\kone\bigdot

\title{Local limit theorem for joint subgraph counts}
\author[Ashwin Sah]{Ashwin Sah\nfts{1}}
\address{\nfts{1}Department of Mathematics, Massachusetts Institute of Technology, Cambridge, MA 02139, USA}
\email{asah@mit.edu}
\author[Mehtaab Sawhney]{Mehtaab Sawhney\nfts{2}}
\address{\nfts{2}Department of Mathematics, Columbia University, New York, NY 10027, USA}
\email{m.sawhney@columbia.edu}
\author[Daniel G. Zhu]{Daniel G. Zhu\nfts{3}}
\address{\nfts{3}Department of Mathematics, Princeton University, Princeton, NJ 08544, USA}
\email{zhd@princeton.edu}
\thanks{AS and MS were supported by NSF Graduate Research Fellowship Program DGE-2141064. A portion of this research was conducted during the period MS served as a Clay Research Fellow. DZ is supported by the NSF Graduate
Research Fellowship Program DGE-2039656.}

\begin{document}
\begin{abstract}
Extending a previous result of the first two authors, we prove a local limit theorem for the joint distribution of subgraph counts in the Erd\H{o}s-R\'{e}nyi random graph $G(n,p)$. This limit can be described as a nonlinear transformation of a multivariate normal distribution, where the components of the multivariate normal correspond to the graph factors of Janson. As an application, we show a number of results concerning the existence and enumeration of proportional graphs and related concepts, answering various questions of Janson and collaborators in the affirmative.
\end{abstract}
\maketitle
\section{Introduction}\label{sec:intro}
For a fixed graph $H$ and probability $p \in (0,1)$, the count $X_H$ of subgraphs isomorphic to $H$ in the Erd\H{o}s-R\'{e}nyi random graph $G(n,p)$ has been one of the most frequently studied random variables in the theory of random graphs. While a central limit theorem for $X_H$, establishing that it, properly normalized, converges to a normal distribution in the limit $n \to \infty$, was established by Nowicki in 1985 \cite{Now89} (see also \cite{NW88,Ruc88}), a series of recent works has focused on the local aspects of the distribution of $X_H$, with the aim of controlling the point probabilities $\setp[X_H = x]$ (see \cite{FKS21} and references therein). Chief among these results are a series of local limit theorems \cite{GK16,Ber16,Ber18,SS22}, culminating in the following local limit theorem of the first two authors for all connected subgraph counts:
\begin{thm}[{\cite[Theorem~1.1]{SS22}}]
Let $H$ be a connected graph with at least two vertices and let $p \in (\lambda, 1-\lambda)$. Then, for every $\eps > 0$ and $x \in \setz$, we have
\[\abs*{\sigma \setp[X_H = x] - \caln\paren*{\frac{x-\sete[X_H]}{\sigma}}} \lsim_{H,\lambda,\eps} n^{-1/2+\eps},\]
where $\sigma^2 = \Var[X_H]$ and $\caln(z) = (2\pi)^{-1/2} e^{-z^2/2}$ is the standard normal probability distribution function.
\end{thm}

The main objective of this paper is to develop a local limit theorem for joint distributions of connected subgraph counts, controlling probabilities of the form $\setp[X_H = x_H\;\forall H \in \calh]$ for certain sets of connected graphs $\calh$. Such a result would be applicable far beyond connected subgraph counts, as every graph statistic counting structures of bounded size,\footnote{Formally, such a graph statistic must be vertex-symmetric and satisfy the condition that, for every $n$, it can be written as a polynomial in the indicator functions of the edges of the graph with degree bounded independently of $n$.} including subgraph counts of disconnected graphs and induced subgraph counts, can be written as a function of $n$ and a finite number of connected subgraph counts (see \cref{sec:arith} of this paper for more details).

From a purely distributional standpoint, the behavior of $(X_H)_{H \in \calh}$ was settled (albeit somewhat unsatisfyingly) by Nowicki \cite{Now89}, who observed that for large $n$, almost all the variation in $X_H$ is controlled by the edge count of the graph, thus implying that in the $n\to\infty$ limit, any finite collection of subgraph counts converges in distribution to a collection of perfectly correlated normal random variables. In order to study graph statistics on $G(n,p)$ for which the above description is insufficiently precise, Janson \cite{Jan90,Jan94} developed the notion of \vocab{graph factors} $\gamma_H$, analogues of subgraph counts obtained through an orthogonalization process. Specifically, letting $x_e$ be the indicator variable of an edge $e$ and $\chi_e = (x_e - p)/\sqrt{p(1-p)}$ be a normalized version of $x_e$ satisfying $\sete[\chi_e] = 0$ and $\sete[\chi_e^2] = 1$, Janson defined
\[\gamma_H = \sum_{H' \cong H} \prod_{e \in E(H')} \chi_e,\]
where the sum is over all subgraphs $H'$ of $K_n$ isomorphic to $H$. It is straightforward to show that if $H$ and $H'$ have no isolated vertices, then $\sete[\gamma_H \gamma_{H'}]$ is $0$ if $H \ncong H'$ and otherwise equal to the number of subgraphs of $K_n$ isomorphic to $H \cong H'$. Moreover, we will show in \cref{sec:arith} of this paper that if we restrict ourselves to connected $H$, then the $\gamma_H$ are algebraically independent, and that the $X_H$ can be canonically be written as polynomials in the $\gamma_H$, and vice versa.

Janson proved that for any finite collection $\calh$ of (nonisomorphic) connected graphs with at least two vertices,\footnote{If $H$ is a single vertex, $\gamma_H$ is always $n$ and thus is uninteresting in this context.} the distribution of $(\gamma_H)_{H \in \calh}$, after appropriate scaling, converges to a standard multivariate normal distribution, with no correlations between different graph factors. Transforming back to subgraph counts yields a model of the joint distribution of subgraph counts consisting of a polynomial transform of a multivariate normal distribution.

The main result of this paper states that this model is asymptotically correct even at the level of point probabilities.
\begin{thm} \label{thm:main}
Let $p \in (\lambda, 1-\lambda)$ and let $\calh$ be a downwards closed set of nonisomorphic connected graphs with at least two vertices (in the sense of \cref{def:downwards-closed}). Then, for every $\eps > 0$ and permissible $(y_H)_{H \in \calh}$ (in the sense of \cref{def:permissible}) we have
\[\abs*{\setp[\gamma_{H} = y_H\;\forall H\in \calh] \prod_{H \in \calh} \paren*{(p(1-p))^{e(H)/2}\sigma_{H}} - \prod_{H \in \calh} \caln(y_H/\sigma_{H})} \lsim_{ \calh, \lambda, \eps} n^{-1/2+\eps},\]
where $\sigma_{H}^2 = \Var[\gamma_{H}]$ and $\caln(z) = (2\pi)^{-1/2} e^{-z^2/2}$ is the standard normal probability distribution function.
\end{thm}

Two terms in \cref{thm:main}, which will be defined formally in \cref{sec:arith}, merit discussion. First of all, the ``downwards closed'' condition is needed to ensure that there is a bijective mapping between $(\gamma_H)_{H \in \calh}$ and $(X_H)_{H \in \calh}$, and encompasses choices for $\calh$ such as the set of connected graphs $H$ with $2 \leq v(H) \leq a$ and $e(H) \leq b$, for arbitrary positive integers $a \geq 2$ and $b \geq 1$. Second, a tuple $(y_H)_{H \in \calh}$ is permissible if and only if it corresponds to values for $(X_H)_{H \in \calh}$ that are integers; obviously, if $(y_H)_{H \in \calh}$ is not permissible, we have $\setp[\gamma_{H} = y_H\;\forall H\in \calh] = 0$. The set of permissible $(y_H)_{H \in \calh}$ forms a ``skew lattice''\footnote{For the purposes of this introduction, a \emph{skew lattice} is, after possibly permuting coordinates, the image of $\setz^d$ under a ``upper triangular'' map of the form $(x_1,\ldots,x_d) \mapsto (a_1x_1+\psi_1,a_2x_2 + \psi_2(x_1),\ldots,a_dx_d+\psi_d(x_1,\ldots,x_{d-1}))$ for nonzero $a_1,\ldots,a_d$, constant $\psi_1$, and arbitrary functions $\psi_2,\ldots,\psi_d$. The density of such a skew lattice is $1/\abs{a_1a_2 \cdots a_d}$.} with density $(p(1-p))^{\sum_{H \in \calh} e(H)/2}$, explaining the appearance of that term in \cref{thm:main}.

Combined with appropriate tail bounds on the $\gamma_H$, \cref{thm:main} paves the way for a complete description of the local distribution of any graph statistic that can be written in terms of a finite number of connected subgraph counts, including disconnected subgraph counts and induced subgraph counts which may not necessarily have an associated local central limit theorem (see \cite[Theorem~1.3]{SS22}). Moreover, \cref{thm:main} locally controls the joint distribution of any collection of such graph statistics; in particular, since edge count is a subgraph count, it allows for local limit theorems to be proven in the $G(n,m)$ model as well.

Another application of \cref{thm:main} is to show the existence of graphs with exactly specified subgraph counts, as it implies that given any integer tuple $(x_H)_{H \in \calh}$ depending on $n$, the probability that $X_H = x_H$ for all $H$ is positive for sufficiently large $n$, provided that corresponding values of $(\gamma_H/\sigma_H)_{H \in \calh}$ are bounded. In \cref{sec:prop}, we apply this to the theory of proportional graphs, developed to describe the different asymptotic behaviors of the induced subgraph count $Y_H$ of a graph $H$ in the $G(n,p)$ and $G(n,m)$ models. Answering a question of Janson \cite{Jan94}, we show that for almost all\footnote{though not cofinitely many} rational $p \in (0,1)$ (the characterization of which we explicitly state), there exist infinitely many $H$ such that the variable $Y_H$ in the $G(n,\floor{p\binom{n}{2}})$ model satisfies $\Var[Y_H] \asymp n^{2v(H)-6}$ and converges to a nonnormal distribution. Interestingly, the integrality constraints translate to nontrivial number theory; for instance, if we set $p = \frac{1}{2}$, every such $H$ must have more than $10^{390}$ vertices!

\subsection*{Proof overview}
\subsubsection*{Fourier analysis, Stein's method, and decoupling}
The proof of \cref{thm:main} is Fourier-analytic in nature and follows the same basic outline as \cite{SS22} (which in turn generalized techniques from \cite{Ber18}). After applying Fourier inversion, we aim to show a multivariate characteristic function $\varphi^\calf_X(\vt)$ (whose precise definition we will postpone) involving the $X_H$ is close to $\varphi^\calf_Z(\vt)$, the characteristic function of a transformed multivariate normal. To accomplish this, we use two classes of techniques depending on the size of $\vt$. At small frequencies, we use a multivariate generalization of Stein's method of exchangeable pairs due to Meckes \cite{M09} to show that the $\gamma_H$ are well-approximated by independent normal random variables. At larger $\vt$, where $\varphi^\calf_Z(\vt)$ becomes negligibly small, we use a decoupling trick of Berkowitz \cite{Ber18} to show that $\varphi^\calf_X(\vt)$ is similarly small. Numerous different setups are needed to handle various cases concerning the sizes of the components of $\vt$, but as in \cite{SS22}, there is a fundamental distinction between $\vt$ of intermediate-size and $\vt$ that are extremely large.

\subsubsection*{Integral factor systems}
A major step within the proof is to identify the appropriate lattice on which to apply Fourier inversion. If one uses $(X_H)_{H \in \calh}$, the aforementioned correlations between the different $X_H$ make it so that the characteristic function is no longer concentrated sufficiently near the origin, rendering our proof method useless. On the other hand, the space of permissible $(\gamma_H)_{H \in \calh}$, where these correlations have been removed, is not a lattice at all! To resolve this issue, we introduce what we call an \vocab{integral factor system} $\calf = (F_H)_{H \in \calh}$, which behaves as a hybrid of $(X_H)$ and $(\gamma_H)$. Like $(\gamma_H)$, most of the correlations between the $X_H$ have been removed, making the characteristic function well-behaved. Moreover, the set of $(F_H)$ corresponding to integral $(X_H)$ is precisely $\setz^\calh$, which allows Fourier inversion to be used. For further discussion motivating integral factor systems, see \cref{subsec:ifsmotiv}. We believe the underlying techniques will prove useful in many situations where one seeks to establish a multivariate local limit theorem. 

\subsubsection*{Polynomial bases for graph statistics}
To facilitate the conversion between the $X_H$, $F_H$, and $\gamma_H$, we introduce a number of closely related algebraic objects that keep track of the relations between these statistics, which can be viewed as either a discrete analogue of the flag algebras of Razborov \cite{Raz07}, or a graph-theoretic analogue of the ring of symmetric functions. Just as there are many standard bases for symmetric functions (such as $m_\lambda$, $p_\lambda$, $e_\lambda$, and $s_\lambda$), we find several different bases (such as $X_H$ and $\gamma_H$) of these objects, which are all used at different points in the proof.

\subsection*{Outline}
In \cref{sec:arith}, we discuss the algebraic relations between the $X_H$ and $\gamma_H$ and construct integral factor systems. In \cref{sec:fourier} we set up the main Fourier inversion, while we control the characteristic function in \cref{sec:distrib} for small $\vt$ and in \cref{sec:decouple} for large $\vt$. In \cref{sec:prop} we discuss applications to proportional graphs and related concepts.

\subsection*{Notation and conventions}
All graphs considered in this paper are finite and simple. Given a graph $G$, we let $V(G)$ and $E(G)$ denote the vertices and edges of $G$, respectively, and let $v(G)$, $e(G)$, and $\aut G$ denote the number of vertices, edges, and automorphisms of $G$. We let $\sqcup$ denote the disjoint union of graphs, $kG$ denote the disjoint union of $k$ copies of $G$, and $\bar G$ denote the complement of $G$. We let $K_n$, $K_{m,n}$, and $P_n$ denote the complete, complete bipartite, and path graph on $n$ vertices, $m+n$ vertices, and $n$ edges, respectively. We consider the empty graph $K_0$ to be disconnected and denote it using the symbol $\varnothing$. We also let $\kone$ be a shorthand for $K_1$. Finally, given a graph $H$ and partition $\calp$ of its vertices such that no vertices in the same part are connected by an edge, we define $H/\calp$ to be a graph on $\calp$, with two partitions connected if there is at least one edge of $H$ between their individual vertex sets.

We take $f \lsim g$ to mean $f = O(g)$, with subscripts on either the $\lsim$ or the $O$ denoting dependence in the explicit constants. We let $f \asymp g$ mean $f = \Theta(g)$, with subscripts treated similarly. We define $[n] = \set{1,2,\ldots,n}$, which we at times identify with the vertices of $K_n$ and $G(n,p)$. When considering some $0 < p < 1$, we will always have $p \in (\lambda, 1-\lambda)$. Whenever considering a set of graphs $\calh$, the letter $\ell$ will always denote the maximum number of vertices of a graph in $\calh$. Note that an asymptotic with constant depending on $\ell$ is equivalent to an asymptotic with constant depending on $\calh$, as for each value of $\ell$ there are only finitely many possibilities for $\calh$.

\section{Subgraph Count Arithmetic} \label{sec:arith}
\subsection{Polynomial relations between subgraph count statistics}
In this subsection, we will treat $X_H$ and $\gamma_H$ as functions $\set{\text{graphs}} \to \setr$. We will also define the following other statistics:
\begin{defn}
For a graph $H$ let $\tilde X_H = \aut H \cdot X_H$ be the number of injective homomorphisms from $H$ into a given graph. If $H_1,\ldots,H_m$ are the connected components of $H$, then let $X^*_H = \prod_{i \in [m]} X_{H_i}$ and $\tilde X^*_H = \prod_{i \in [m]} \tilde X_{H_i}$. Define $\tilde \gamma_H$, $\gamma^*_H$, and $\tilde \gamma^*_H$ similarly.
\end{defn}
\begin{rmk}
It is not in general true that $\tilde X^*_H = \aut H \cdot X^*_H$, since $\aut(H_1 \sqcup H_2)$ is not necessarily equal to $\aut H_1 \cdot \aut H_2$.
\end{rmk}
\begin{defn}
Given two graphs $H_1$ and $H_2$, write $H_1 \preceq H_2$ if $H_1$ can be obtained from $H_2$ via a combination of edge deletion, vertex deletion, and merging disconnected vertices. Write $H_1 \prec H_2$ if $H_1 \preceq H_2$ but $H_1 \neq H_2$.
\end{defn}
This is clearly a partial order on graphs.
\begin{prop} \label{prop:linrel}
Let $\mu_H$ and $\nu_H$ denote any two of
\[X_H,\quad X_H^*,\quad (p(1-p))^{e(H)/2}\gamma_H,\quad\text{or } (p(1-p))^{e(H)/2}\gamma^*_H.\]
Then, for any $H$, one can write $\nu_H$ as a linear combination $\sum_{H'\preceq H} a_{H'} \mu_{H'}$, where $a_{H} \asymp_H 1$, $a_{H'} \lsim_H 1$ for all $H' \preceq H$, and $a_H =1$ if $H$ is connected. (Note that these bounds are independent of $p$.)
\end{prop}
\begin{proof}
It is not difficult to see that the condition in \cref{prop:linrel} between $\mu_H$ and $\nu_H$ is an equivalence relation. Therefore it suffices to check the condition three times for suitable $\mu_H$ and $\nu_H$.

First we consider $\mu_H = X_H$ and $\nu_H = X^*_H$. If $H$ is connected, we have $X^*_H = X_H$, so we are done. Otherwise, letting $H_1,\ldots,H_m$ denote the connected components of $H$, we claim that
\[\tilde X^*_H = \sum_{\calp} \tilde X_{H/\calp},\]
where $\calp$ runs over all set partitions of $V(H)$ where all elements of $V(H_i)$ are in different parts for all $i$. Indeed, the left-hand side, evaluated on a given $G$, counts the number of graph homomorphisms $H \to G$ that are injective on each $H_i$. For each such homomorphism, there is a unique such set partition describing which vertices of $H$ map to the same vertex of $G$, and the number of homomorphisms corresponding to a given set partition $\calp$ can be easily shown to be exactly $\tilde X_{H/\calp}$. Note that $H/\calp \preceq H$ for all $\calp$, and equality holds if and only if all the components of $\calp$ are singletons. Converting from the $\tilde X_H$ and $\tilde X_H^*$ to $X_H$ and $X_H^*$, we get the desired linear relation where all coefficients are independent of $p$ (and hence certainly $O_{H}(1)$) and the coefficient of $H$ is positive.

Now we consider $\mu_H = (p(1-p))^{e(H)/2}\gamma_H$ and $\nu_H = X_H$. Since $x_e = \sqrt{p(1-p)}\chi_e + p$, we conclude that
\[
\tilde X_H = \sum_{H' \subseteq H} p^{e(H) - e(H')} (p(1-p))^{e(H')/2} \tilde \gamma_{H'}
\]
and hence
\begin{equation} \label{eq:xgamma}
X_H = \sum_{H' \subseteq H} \frac{\aut H'}{\aut H} p^{e(H) - e(H')} (p(1-p))^{e(H')/2} \gamma_{H'}.
\end{equation}

Finally, we consider $\mu_H = X^*_H$ and $\nu_H = (p(1-p))^{e(H)/2}\gamma^*_H$. Here, note that by combining the previous two cases we may construct identities of the form $(p(1-p))^{e(H)/2}\gamma_H = \sum_{H' \preceq H} a_{H'} X^*_H$. Multiplying these identities for various connected $H$ yields the desired. (Here we have used the fact that if $H'_1 \preceq H_1$ and $H'_2 \preceq H_2$, then $H'_1 \sqcup H'_2 \preceq H_1 \sqcup H_2$.)
\end{proof}

We now construct the following object to encapsulate all these relations:
\begin{defn}
Let $\calr$ be the $\setr$-vector space generated by the $X_H$.
\end{defn}
\begin{cor}
The sets $\set{X_H}$, $\set{X^*_H}$, $\set{\gamma_H}$, and $\set{\gamma^*_H}$ are all bases of $\calr$. Moreover, as rings,
\[\calr = \setr[X_H \colon \text{\upshape $H$ connected}] = \setr[\gamma_H \colon \text{\upshape $H$ connected}].\]
\end{cor}
\begin{proof}
The functions $X_H$ are linearly independent, since $X_H$ is $1$ on the graph $H$ but $0$ on any graph $H' \prec H$. Thus $\set{X_H}$ is a basis of $\calr$. The fact that $\set{X_H^*}$, $\set{\gamma_H}$, and $\set{\gamma_H^*}$ are bases as well follows from \cref{prop:linrel}. The fact that $\calr = \setr[X_H : \text{$H$ connected}]$ (resp.~$\setr[\gamma_H \colon \text{$H$ connected}]$) is an alternate way of saying that the $X_H^*$ (resp.~$\gamma_H^*$) form a basis of $\calr$. 
\end{proof}
In particular, the fact that $\calr = \setr[X_H \colon \text{\upshape $H$ connected}]$ implies that $\calr$ is an $\setr$-algebra.

At this point, we may now define the terms used in \cref{thm:main}.

\begin{defn}\label{def:downwards-closed}
A finite set of non-$\kone$ connected graphs $\calh$ is \emph{downwards closed} if for all $H \in \calh$ and $H' \preceq H$ with $H'$ connected and not $\kone$, we have $H' \in \calh$. In this case, it is true that if $H' \preceq H \in \calh$, all components of $H'$ are in $\calh \cup \set{\kone}$; thus \cref{prop:linrel} implies that
\[\setr[X_H \colon H \in \calh \cup \set{\kone}] = \setr[\gamma_H \colon H \in \calh \cup \set{\kone}].\]
Call this algebra $\calr_{\calh}$.
\end{defn}
\begin{rmk}
There does not appear to be a nice basis of $\calr_\calh$ in terms of the unstarred $X_H$ or $\gamma_H$.
\end{rmk}
Now, in $\calr$, we have $X_\kone = \gamma_\kone$, and both evaluate to the number of vertices of a given graph. Thus, if we are working with graphs with a fixed number of vertices $n$, it makes sense to consider the following object.
\begin{defn}
For an integer $n$, let $\calr_n = \calr/(\gamma_\kone - n)$ and for downwards closed $\calh$ let $\calr_{\calh,n} = \calr_\calh/(\gamma_\kone - n)$. Note that the intersection of the ideal $(\gamma_\kone - n) \subseteq \calr$ with $\calr_\calh$ is the ideal generated $(\gamma_\kone - n)$ in $\calr_\calh$, so we may treat $\calr_{\calh,n}$ as a subalgebra of $\calr_n$.
\end{defn}
\begin{rmk}
Since $\calr$ was defined as a subalgebra of the algebra of functions $\set{\text{graphs}} \to \setr$, a relation between statistics in $\calr$ can be verified by checking it for all graphs. However, this is not the case for $\calr_n$; there are elements that evaluate to zero on all $n$-vertex graphs (such as $X_H$ for $H$ with more than $n$ vertices) but are not actually zero.
\end{rmk}
We may now define the notion of permissibility as being values of $(\gamma_H)_{H \in \calh}$ ``corresponding'' to integral subgraph counts, in the following manner.
\begin{defn}\label{def:permissible}
An \vocab{evaluation} is an algebra homomorphism $\Phi \colon R_{\calh,n} \to \setr$. Call an evaluation $\Phi$ \vocab{integral} if $\Phi(X_H) \in \setz$ for all $H \in \calh$. For a given integer $n$ and downwards closed $\calh$, a tuple $(y_H)_{H \in \calh}$ is \vocab{permissible} if there exists an integral evaluation $\Phi$ on $\calr_{\calh,n}$ such that $\Phi(\gamma_H) = y_H$ for all $H \in \calh$.
\end{defn}

Finally, as we will frequently work with graphs with no isolated vertices, we abbreviate ``graph with no isolated vertices'' to ``\vocab{NIV} graph''.

\subsection{Motivational remarks on integral factor systems} \label{subsec:ifsmotiv}
In order to prove \cref{thm:main}, we will use Fourier inversion, deriving it from the result that a characteristic function based on the $X_H$ is close to a characteristic function based on a multivariate normal. However, the appropriate lattice to use is far from obvious. To illuminate the types of issues that can arise, we begin with a motivational note.

As a toy example, let $(X_1,X_2)$ be integer-valued random variables (depending on $n$) and suppose we wish to show that $(X_1,X_2)$ is locally close in distribution to $(nZ_1, n^2Z_2 + f(Z_1))$, where $Z_1$ and $Z_2$ are independent standard normal random variables and $f$ is an arbitrary, but reasonably well-behaved, function (also depending on $n$). Since the map $(z_1, z_2)\mapsto (nz_1, n^2z_2 + f(z_1))$ has Jacobian determinant $n^3$ for differentiable $f$, we formally wish to show that
\[
\setp[(X_1,X_2)=(x_1,x_2)] = \frac{\caln(\frac{x_1}{n})\caln(\frac{x_2-f(x_1/n)}{n^2}) + o(1)}{n^3}.
\]

A na\"{i}ve attempt at Fourier inversion would see us define
\[\varphi_X(t_1,t_2) = \sete[e^{i(t_1X_1+t_2X_2)}]\quad\text{and}\quad\varphi_Z(t_1,t_2) = \sete[e^{i(t_1\cdot nZ_1+t_2(n^2Z_2+f(Z_1)))}],\]
so that
\begin{align*}
\setp[(X_1,X_2)=(x_1,x_2)] &= \frac{1}{(2\pi)^2} \int_{[-\pi,\pi]^2} e^{-i(t_1x_1+t_2x_2)}\varphi_X(\vt)\,d\vt \\
\intertext{and}
\frac{\caln(\frac{x_1}{n})\caln(\frac{x_2-f(x_1/n)}{n^2})}{n^3} &= \frac{1}{(2\pi)^2} \int_{\setr^2} e^{-i(t_1x_1+t_2x_2)} \varphi_Z(\vt)\,d\vt.
\end{align*}
Thus, we will be done if we can show that $\varphi_X(\vt)$ and $\varphi_Z(\vt)$ are sufficiently close.

\begin{figure}
\subcaptionbox{\label{fig:norm}}{%
\begin{tikzpicture}[scale=2,every node/.style={font={\scriptsize}}]
\draw (-1,-1) rectangle (1,1);
\fill (0,0) ellipse (0.4 and 0.1);
\draw[|-|,yshift=0.2cm] (-0.4,0)-- node[anchor=south]{$\asymp n^{-1}$}(0.4,0);
\draw[|-|,xshift=0.5cm] (0,-0.1)-- node[anchor=west]{${\asymp} n^{-2}$}(0,0.1);
\end{tikzpicture}}\quad
\subcaptionbox{\label{fig:shear}}{%
\begin{tikzpicture}[scale=2,every node/.style={font={\scriptsize}}]
\begin{scope}
\clip (-1,-1) rectangle (1,1);
\fill[rotate=-3.5] (0,0) ellipse (1.6 and 0.025);
\draw[|-|] (0.9,-0.025)--(0.9,0.025) node[anchor=south east] at (1.05,0.025){${\asymp} n^{-3}$};
\end{scope}
\draw (-1,-1) rectangle (1,1);
\draw [densely dotted] (-1,-0.025) -- (1,-0.025) (-1,0.025) -- (1,0.025);
\end{tikzpicture}}\quad
\subcaptionbox{\label{fig:intermed}}{%
\begin{tikzpicture}[scale=2,every node/.style={font={\scriptsize}}]
\fill[lightgray] (-1,-0.1) rectangle (1,0.1);
\fill (0,0) ellipse (0.4 and 0.025);
\draw (-1,-1) rectangle (1,1);
\draw[|-|,yshift=0.2cm] (-0.4,0)-- node[anchor=south]{$\asymp n^{-1}$}(0.4,0);
\draw[|-|,xshift=0.5cm] (0,-0.025)-- node[anchor=west]{${\asymp} n^{-3}$}(0,0.025);
\draw[|-|,xshift=-0.5cm] (0,-0.1)-- node[anchor=east]{${\asymp} n^{-2}$}(0,0.1);
\end{tikzpicture}}%
\caption{Conceptual drawings of $\varphi_Z(\vt)$ on $[-\pi,\pi]^2$ for \subref{fig:norm} $f(z) = 0$, \subref{fig:shear} $f(z) = \alpha n^3 z$, and \subref{fig:intermed} $f(z) = \alpha n^3 z^2$, for some constant $\alpha \asymp 1$. Black areas indicate where $\varphi_Z(\vt)$ is of constant order, gray areas indicate where $\varphi_Z(\vt)$ is polynomially small, and white areas indicate where $\varphi_Z(\vt)$ is superpolynomially small.} \label{fig:cartoons}
\end{figure}
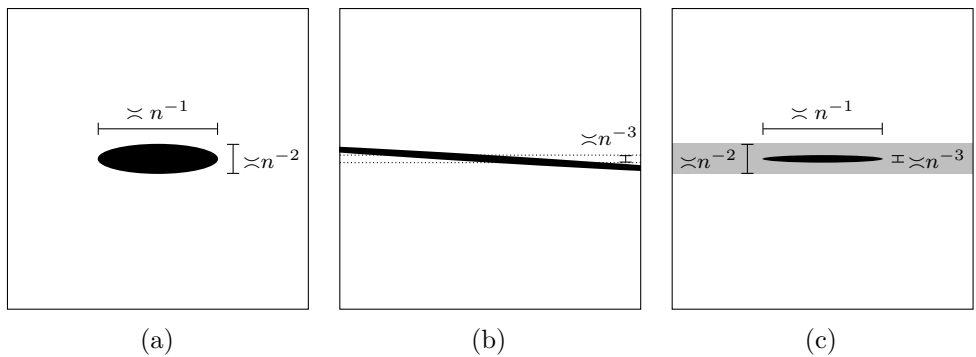

To investigate the problems that can occur, we will describe $\varphi_Z(\vt)$ for various choices of $f$. First of all, in the case where $f(z) = 0$, we simply have $\varphi_Z(\vt) = e^{-((nt_1)^2 + (n^2t_2)^2)/2}$, which is concentrated in an region of area $\asymp n^{-3}$ near the origin; this situation is shown in \cref{fig:norm}. If one could show $\abs{\varphi_X(\vt) - \varphi_Z(\vt)} = o(1)$ near the origin and $\varphi_X(\vt) = n^{-\omega(1)}$ elsewhere in $[-\pi,\pi]^2$, the local limit theorem would then be proved.

A complication occurs, however, if we have $f(z) = \alpha n^3 z$, for some constant $\alpha \asymp 1$. In this case,
\[\varphi_Z(\vt) = e^{-((nt_1+\alpha n^3 t_2)^2 + (n^2t_2)^2)/2},\]
a sheared version of the previous case; this situation is shown in \cref{fig:shear}. While the characteristic function is still well-concentrated in a region of area $n^{-3}$, the region in question is roughly a slanted rectangle of size $\Theta(1) \times \Theta(n^{-3})$, which notably escapes the square $[-\pi,\pi]^2$. In this case, our goal would be to show that $\varphi_X(\vt)$ is concentrated on the image of this box under the ``mod $2\pi$'' reduction map $\setr^2 \to [-\pi,\pi)^2$, and we would have to carefully construct a correspondence between values of $\vt$ and $\vt'$ with $\vt \equiv \vt' \pmod{2\pi}$ where we can prove $\varphi_X(\vt) = \varphi_Z(\vt') + o(1)$.

While this is already somewhat of a headache, a more serious problem arises if we introduce significant nonlinearity into $f$; for example, by letting $f(z) = \alpha n^3 z^2$ for some $\alpha \asymp 1$. In this case,
\[\varphi_Z(\vt) = e^{-(n^2t_2)^2/2} \cdot \sete[e^{i(nt_1 \cdot Z_1 + \alpha n^3 t_2 \cdot Z_1^2)}] = e^{-(n^2t_2)^2/2} \frac{e^{-n^2 t_1^2/(2-4i\alpha n^3 t_2)}}{\sqrt{1-2i\alpha n^3 t_2}},\]
which has magnitude
\[\exp\paren*{-\frac{1}{2}\paren*{n^4t_2^2 + \frac{n^2t_1^2}{1 + 4\alpha^2n^6t_2^2}}} \Big/ \paren*{1 + 4\alpha^2n^6t_2^2}^{1/4}.\]
This function is of constant order on a $\Theta(n^{-1}) \times \Theta(n^{-3})$ rectangle centered at the origin, but unlike the previous cases, it does not immediately rapidly decay, and there is now a new intermediate regime in a box of size $\Theta(1) \times \Theta(n^{-2})$ where the characteristic function is of size $\asymp n^{-1/2}$. This situation is shown in \cref{fig:intermed}. For our purposes, this is incredibly problematic, as to prove the local limit theorem one would now need bounds of the form $\abs{\varphi_X(\vt) - \varphi_Z(\vt)} = o(n^{-1})$ in the intermediate regime due to the magnitude of its area, which are implausibly strong.

To motivate the solution to these problems, we make the observation that, fundamentally, all our issues arose because $f(z)$ became much larger than $n^2$; indeed, if instead we had $f(z) = \alpha n^2 z$ or $f(z) = \alpha n^2 z^2$, one can show that although $\varphi_Z(\vt)$ would undergo some distortion, the basic picture would remain the same as the $f(z) = 0$ case, with large values in a $\Theta(n^{-1}) \times \Theta(n^{-2})$ rectangle near the origin and rapid decay away from this rectangle. To force this to happen, we transform $(X_1,X_2)$ to a different set of integer-valued random variables $(F_1,F_2)$, given by $F_1 = X_1$ and $F_2 = X_2 - g(X_1)$, where $g$ sends integers to integers. It is clear that the local limit theorem we wish to prove is equivalent to proving that $(F_1,F_2)$ is locally close in distribution to $(nZ_1, n^2Z_2 + f^*(Z_1))$, where $f^*(z) = f(z) - g(nz)$. In our examples, we may choose $g(x) = \floor{n^2\alpha}x$ and $g(x) = \floor{n\alpha}x^2$, respectively, so that $f^*(z) = n\set{n^2\alpha} z$ and $f^*(z) = n^2 \set{n\alpha} z^2$, which are sufficiently small for our purposes.

We remark that the fact that such nice choices of $g$ were available cannot be taken for granted; if we instead had to prove that $(X_1, X_2)$ was close to $(n^2 Z_1, n^3 Z_2 + \alpha n^4 Z_1^2)$, then we would need $g(n^2 z) = \alpha (n^2 z)^2 + O(n^3)$, and no integer polynomial $g$ satisfies this condition for noninteger $\alpha$. Any choice of $g$ in this case would likely be difficult to analyze, and it is doubtful that our methods would easily transfer. Thankfully, in the case of subgraph counts it is always the case that a polynomial transform will suffice.

\subsection{Integral factor systems}
We now define an integral factor system.

\begin{defn} \label{def:ifs}
Given downwards closed $\calh$, an integer $n$, and some $\eta > 0$, a collection of statistics $\calf = (F_H)_{H \in \calh}$, where $F_H \in \calr_{\calh,n}$ for all $H \in \calh$, is an \vocab{$(\calh, n, \eta)$-integral factor system} (abbreviated \vocab{$(\calh, n, \eta)$-IFS} or \vocab{IFS} if the parameters are unimportant) if for all $H \in \calh$, in $\calr_n$ we can write
\[F_H = \sum_{\NIV H' \preceq H} a_{H'} X^*_{H'} = \sum_{\NIV H' \preceq H} b_{H'} \gamma^*_{H'} = \sum_{\NIV H' \preceq H} c_{H'} \gamma_{H'},\]
where
\begin{cond}
\item \label{cond:ifs1} all $a_{H'}$ are integers and $a_H = 1$;
\item \label{cond:ifs2} $\abs{b_{H'}} \leq \eta n^{(v(H) - v(H'))/2}$ for all $H'$ and $b_H = (p(1-p))^{e(H)/2}$;
\item \label{cond:ifs3} $\abs{c_{H'}} \leq \eta n^{(v(H) - v(H'))/2}$ for all $H'$ and $c_H = (p(1-p))^{e(H)/2}$.
\end{cond}
\end{defn}

\begin{lem} \label{lem:ifsexist}
For every $\calh$ and $n$, an $(\calh, n, \eta)$-IFS exists with $\eta \lsim_{\lambda,\ell}1$.
\end{lem}

The proof of this lemma proceeds in two steps. We first prove it ignoring \cref{cond:ifs3} of \cref{def:ifs}. Then we prove that in fact, any $F_H$ satisfying \cref{cond:ifs2} must also satisfy \labelcref{cond:ifs3}, up to multiplying $\eta$ by a constant depending on $\lambda$ and $\ell$.

\begin{proof}[Proof of \cref{lem:ifsexist} ignoring \cref{cond:ifs3}]
We claim that in this case we can actually get $\abs{b_{H'}} \leq 1$ always. To start, observe that if $H$ consists of the disjoint union of an NIV graph $H'$ and $m$ isolated vertices, then $\gamma^*_{H} = n^m \gamma^*_{H'}$ in $\calr_n$. Therefore, by applying \cref{prop:linrel} with $\mu_H = \gamma^*_H$ and $\nu_H = X_H$ and multiplying such relations, we find that for any NIV graph $H$, in $\calr_n$ we can write $X^*_H$ as a linear combination of $\gamma^*_{H'}$ for NIV $H' \preceq H$ such that the coefficient of $\gamma^*_H$ is $(p(1-p))^{e(H)/2} < 1$.

We now describe a process to construct $F_H$. Start with the identity
\[X^*_H = \sum_{\NIV H'\preceq H} s_{H'} \gamma^*_{H'}.\]
While there exists some $H'$ with $\abs{s_{H'}} > 1$, choose one that is maximal under $\preceq$, and add a suitable integer multiple of an identity of the form
\[X^*_{H'} = \sum_{\NIV H''\preceq H'} s'_{H''} \gamma^*_{H''}\]
to make $\abs{s_{H'}} \leq 1$, which is possible since $s'_{H'}$ is positive but less than $1$. Note that this process must terminate, since the only coefficients that can be pushed outside of $[-1,1]$ are the $s_{H''}$ with $H'' \prec H'$. Also, we never have $H' = H$, since we had $\abs{s_{H'}} \leq 1$ at the start; thus the leading terms on both sides remain $X_H$ and $(p(1-p))^{e(H)/2} \gamma_H$, respectively. In the end, we find that an integer linear combination of $X^*_{H'}$ is equal to a linear combination of $\gamma^*_{H'}$, where all the coefficents on the right-hand side have magnitude bounded by $1$. We set $F_H$ to be this quantity.
\end{proof}

To finish the proof, we first need a small lemma.
\begin{lem} \label{lem:iv}
Let $H$ be an NIV graph, and suppose $\gamma^*_H = \sum_{H'\preceq H} a_{H'}\gamma_{H'}$ in $\calr$. Then, if $a_{H'} \neq 0$, then $H'$ has at most $v(H) - v(H')$ isolated vertices.
\end{lem}
\begin{proof}
Let $H''$ be an NIV graph, and consider $\sete_{G(n,p)}[\gamma_{H'} \gamma_{H''}]$ for arbitrary $H'$ with at most $\ell$ vertices. This is zero unless $H'$ is the disjoint union of $H''$ and some number $m$ of isolated vertices, in which case it is precisely
\[\frac{v(H'')!}{\aut H''} \binom{n}{v(H'')} \binom{n-v(H'')}{m} \asymp n^{v(H'')+m}.\]
Thus, it follows that if $\sete_{G(n,p)}[\gamma^*_H \gamma_{H''}] \lsim n^{v(H'')+m}$, then we must have $a_{H'} = 0$ for any $H'$ that is the disjoint union of $H''$ and more than $m$ isolated vertices. We now claim that $\sete_{G(n,p)}[\gamma^*_H \gamma_{H''}] \lsim n^{(v(H'')+v(H))/2}$, which will finish since any $H'$ equal to the disjoint union of $H''$ and $m$ isolated vertices with $a_{H'} \neq 0$ must satisfy
\[m \leq \frac{v(H) - v(H'')}{2} \iff m \leq v(H) - v(H'') - m = v(H) - v(H').\]

Expanding out $\gamma_H^*$, it suffices to show that for NIV graphs $H_1,\ldots,H_m$, we have \[\sete_{G(n,p)}[\gamma_{H_1} \cdots \gamma_{H_m}] \lsim n^{(v(H_1) + \cdots v(H_m))/2}.\] To see this, note that if we expand out $\gamma_{H_1} \cdots \gamma_{H_m}$ in terms of the $\chi_e$, we get a sum of many terms, indexed by copies of $H_1,\ldots,H_m$ in an $n$-vertex complete graph. Each such term is bounded in terms of $p$, and has zero expectation if some edge is covered only once. Since no $H_i$ has an isolated vertex, no edge being covered once implies that no vertex is covered once, meaning that every term with nonzero expectation uses at most $\frac{1}{2}(v(H_1) + \cdots + v(H_m))$ vertices. Furthermore, there are $O_{H_i}(1)$ possible overlap patterns given a choice of these vertices. Thus there are $O_{H_i}(n^{(v(H_1) + \cdots v(H_m))/2})$ terms with nonzero expectation, concluding the proof.
\end{proof}

Given \cref{lem:iv}, the finish is relatively straightforward:
\begin{proof}[Proof that \cref{cond:ifs2} implies \labelcref{cond:ifs3} in \cref{def:ifs}]
By \cref{prop:linrel}, we may write each $\gamma^*_{H'}$ as a sum $\sum_{H'' \preceq H'} s_{H''} \gamma_{H''}$ where $s_{H''} \lsim_{\lambda,\ell} 1$. We may ignore terms with $s_{H''} = 0$. If $H''$ is the disjoint union of an NIV graph $H^*$ and $m$ isolated vertices, then we have
\[\gamma_{H''} = \binom{n - v(H^*)}{m} \gamma_{H^*},\]
which, by \cref{lem:iv}, is in fact $O_{\ell}(n^{\frac{1}{2}(v(H') - v(H^*))}) \gamma_{H^*}$. Combining terms using \cref{cond:ifs2} yields the desired expression of $F_H = \sum_{\NIV H'\preceq H} c_{H'} \gamma_{H'}$ with $c_{H'} \lsim_{\lambda,\ell} n^{\frac{1}{2}(v(H) - v(H'))}\eta$. Moreover, since the coefficient of $\gamma_H$ in $\gamma^*_H$ is $1$, we have $c_H = (p(1-p))^{e(H)/2}$.
\end{proof}

\section{Setting up the Main Proof} \label{sec:fourier}
In this section, we use Fourier inversion to reduce \cref{thm:main} to two inequalities on a characteristic function: one (\cref{lem:distrib}) at low frequencies and one (\cref{lem:decouple}) at high frequencies, which are proven in the two subsequent sections.
\subsection{Fourier inversion}
Using \cref{lem:ifsexist}, we take an $(\calh, n, \eta)$-IFS $\calf = (F_H)_{H \in \calh}$ with $\eta \lsim_{\lambda,\ell} 1$. Observe that we may consider each $F_H$ as both a random variable on $G(n,p)$ and a function of the $(\gamma_H)_{H \in \calh}$. Thus, for $\vt = (t_H)_{H \in \calh}$, we may define \[\varphi^\calf_X(\vt) = \sete_{G(n,p)}\brac*{e^{i\sum_H t_H F_H}}\quad\text{and}\quad\varphi^\calf_Z(\vt) = \sete\brac*{e^{i\sum_H t_H F_H(\sigma_{H'} Z_{H'} \colon H' \in \calh)}},\] where the $Z_{H'}$ in the second expectation are independent standard normal random variables.

Consider an integral evaluation $\Phi$. Since $F_H$ over $G(n,p)$ is always integer-valued, by Fourier inversion we have
\[\setp_{G(n,p)}[F_H = \Phi(F_H)\;\forall H\in \calh] = \frac{1}{(2\pi)^{\abs{\calh}}}\int_{[-\pi,\pi]^\calh} e^{-i\sum_H t_H \Phi(F_H)} \varphi_X(\vt)\,d^{\abs{\calh}} \vt.\]
Meanwhile, the probability density of $(F_H(\sigma_{H'} Z_{H'} \colon H' \in \calh))_{H \in \calh}$ at $(\Phi(F_H))_{H \in \calh}$  is both 
\[\frac{1}{(2\pi)^{\abs{\calh}}}\int_{\setr^\calh} e^{-i\sum_H t_H \Phi(F_H)} \varphi_Z(\vt)\,d^{\abs{\calh}} \vt,\]
by Fourier inversion, and
\[\prod_{H \in \calh} \frac{\caln(\Phi(\gamma_H)/\sigma_{H})}{(p(1-p))^{e(H)/2}\sigma_{H}},\]
since the Jacobian determinant of the transform $(\gamma_H)_{H \in \calh} \mapsto (F_H)_{H \in \calh}$ is $\prod_{H \in \calh} (p(1-p))^{e(H)/2}$. As a result
\begin{align*}
\MoveEqLeft \abs*{\setp[\gamma_H = \Phi(\gamma_H) \; \forall H] \prod_{H \in \calh} \paren*{(p(1-p))^{e(H)/2}\sigma_{H}} - \prod_{H \in \calh} \caln(y_H/\sigma_{H})} \\
&= \prod_{H \in \calh} \paren*{(p(1-p))^{e(H)/2}\sigma_{H}} \frac{1}{(2\pi)^{\abs{\calh}}}\abs*{\int_{\setr^\calh} e^{-i\sum_H t_H \Phi(F_H)} (\varphi^\calf_X(\vt) \one_{\vt \in [-\pi,\pi]^\calh} - \varphi^\calf_Z(\vt))\,d^{\abs{\calh}}\vt} \\
&\lsim_{\ell} \prod_{H \in \calh} n^{v(H)/2}\int_{\setr^\calh} \abs{\varphi^\calf_X(\vt) \one_{\vt \in [-\pi,\pi]^\calh} - \varphi^\calf_Z(\vt))}\,d^{\abs{\calh}}\vt.
\end{align*}
To show that this expression is $O_{\lambda,\ell,\eps}(n^{-1/2 + \eps})$ and hence deduce \cref{thm:main}, it suffices to show the following three lemmas:

\begin{lem} \label{lem:distrib}
Let $\calh$ be a downwards closed set of graphs, $n$ be a positive integer, and $\calf$ be an $(\calh, n, \eta)$-IFS. If we let $\eps > 0$ and $I_\eps = \prod_{H \in \calh} [-n^{-v(H)/2+\eps}, n^{-v(H)/2+\eps}]$, then for $\vt \in I_\eps$ we have
\[\abs{\varphi^\calf_X(\vt) - \varphi^\calf_Z(\vt)} \lsim_{\lambda,\ell,\eps,\eta} n^{-1/2+4\eps}.\]
\end{lem}
\begin{lem} \label{lem:decouple}
With $\calh, n, \calf, \eps, I_\eps$ as above, we have that for $\vt \in [-\pi,\pi]^\calh \setminus I_\eps$,
\[\abs{\varphi^\calf_X(\vt)} = n^{-\omega_{\lambda,\ell,\eps,\eta}(1)}.\]
\end{lem}
\begin{lem} \label{lem:smoothness}
With $\calh, n, \calf, \eps, I_\eps$ as above, we have
\[\int_{\setr^\calh \setminus I_\eps} \abs{\varphi^\calf_Z(\vt)}\,d^k \vt = n^{-\omega_{\lambda,\ell,\eps,\eta}(1)}.\]
\end{lem}
\subsection{Local smoothness of transformed normals}
\cref{lem:smoothness} has nothing to do with graph theory, so we dispense with it here. In this section, we normalize the Fourier transform using the convention $\hat f(\vt) = \int_{\setr^d} e^{i\vt \cdot \vx} f(\vx)\,d^d \vx$.
\begin{defn}
Let $d$ be a positive integer and let $a > 0$ be a real number. A \vocab{$(d, a)$-quasishear} is a map $\Psi\colon \setr^d \to \setr^d$ given by
\[\Psi(x_1,\ldots,x_d) = (x_1 + \psi_1, x_2 + \psi_2(x_1), x_3 + \psi_3(x_1,x_2), \cdots, x_d + \psi_d(x_1,\dots,x_{d-1})),\]
where $\abs{\psi_1} \leq a$ and $\psi_i$, for $2 \leq i \leq d$, are polynomials of degree at most $a$ and all coefficients bounded by $a$.
\end{defn}
\begin{prop} \label{prop:qsi}
The inverse of a $(d, a)$-quasishear is a $(d, O_{d,a}(1))$-quasishear.
\end{prop}
\begin{proof}
Given a $(d, a)$-quasishear,
\[(x_1,\ldots,x_d) \mapsto (x_1 + \psi_1, x_2 + \psi_2(x_1), x_3 + \psi_3(x_1,x_2), \cdots, x_d + \psi_d(x_1,\dots,x_{d-1}))\]
its inverse can be defined as $(x_1,\ldots,x_d) \mapsto (y_1,\ldots,y_d)$, where the $y_i$ are defined recursively by letting $y_1 = x_1 - \psi_1$ and $y_i = x_i - \psi_i(y_1,\ldots,y_{i-1})$ for $2 \leq i \leq d$. Observe that each $y_i$ can be written as a finite expression in terms of the $x_i$ and $\psi_i$, so the degrees and coefficients must remain bounded.
\end{proof}
\begin{lem} \label{lem:qsft}
Let $\mathcal{N}\colon \setr^d \to \setr$, given by $\mathcal{N}(\vx) = (2\pi)^{-d/2} e^{-\Abs{\vx}^2_2/2}$, be the standard multivariate normal probability density. Then for any $(d,a)$-quasishear $\Psi$,
\[\abs{\widehat{\mathcal{N} \circ \Psi}(\vt)} = \Abs{\vt}_\infty^{-\omega_{d,a}(1)},\]
where the asymptotic is taken in the limit $\Abs{\vt}_\infty \to \infty$.
\end{lem}
\begin{proof}
Let $r$ be a positive integer. Then for every index $j \in [d]$, we have
\[\abs{t_j}^r \abs{\widehat{\mathcal{N} \circ \Psi}(\vt)} = \abs{(\partial_j^r\suphat{(\mathcal{N} \circ \Psi))}(\vt)} \leq \Abs{\partial_j^r(\mathcal{N} \circ \Psi)}_1.\]

We now claim that $\partial_j^r(\mathcal{N} \circ \Psi)(\vx) = (\mathcal{N} \circ \Psi)(\vx) P(\vx)$ for some polynomial $P(\vx)$ of degree and coefficients that are bounded in terms of $d$, $a$, and $r$. This can be proven by induction on $r$. The $r = 0$ case is trivial, and we note that
\[\partial_j \paren*{(\mathcal{N} \circ \Psi)(\vx) P(\vx)} = (\mathcal{N} \circ \Psi)(\vx) (-(\Psi(\vx) \cdot \partial_j \Psi(\vx))P(\vx) + \partial_j P(\vx)),\]
where we have used the fact that $\nabla \mathcal{N}(\vx) = -\vx \mathcal{N}(\vx)$. As a result, we have
\[\Abs{\partial_j^r(\mathcal{N} \circ \Psi)}_1 = \int_{\setr^d} \abs{P(\vx)} \mathcal{N}(\Psi(\vx))\,d^d \vx = \int_{\setr^d} \abs{P(\Psi\inv(\vx))} \mathcal{N}(\vx)\,d^d \vx.\]
\cref{prop:qsi} implies that $P(\Psi\inv(\vx))$ is also a polynomial with bounded degree and coefficients, so this entire integral is bounded. Since $j$ was arbitrary, we conclude that $\abs{\widehat{\mathcal{N} \circ \Psi}(\vx)} \lsim_{d,a,r}\Abs{\vt}_\infty^{-r}$. Since $r$ was arbitrary, we are done.
\end{proof}

\begin{proof}[Proof of \cref{lem:smoothness}]
Let $\Psi$ be the polynomial map sending
\[\paren*{\frac{\gamma_H}{\sigma_H}}_{H \in \calh} \mapsto \paren*{\frac{F_H}{(p(1-p))^{e(H)/2}\sigma_H}}_{H \in \calh}.\] By \cref{cond:ifs2} of \cref{def:ifs} and the fact that $\sigma_H \asymp_\ell n^{v(H)/2}$, we find that $\Psi$ is an $(\abs{\calh}, O_{\lambda,\ell}(\eta))$-quasishear.

Now, essentially by definition, $\varphi^\calf_Z(\vt) = \widehat{\caln \circ \Psi\inv}(\vt')$, where $\vt' = (t'_H)_{H \in \calh}$ is given by
\[t'_H = (p(1-p))^{e(H)/2}\sigma_H t_H.\]
Thus, by \cref{lem:qsft}, we conclude that there is a function $g$, depending on $\lambda$, $\ell$, and $\eta$, that decays superpolynomially and such that
\[\varphi^\calf_Z(\vt) \leq g\paren*{\max_{H \in \calh} {\abs{t_H} n^{v(H)/2}}}.\]
But then, changing to the integration variable $\tau = \max_{H \in \calh} {\abs{t_H} n^{v(H)/2}}$, we have
\[\int_{\setr^\calh \setminus I_\eps} \abs{\varphi^\calf_Z(\vt)}\,d^k \vt \lsim_\ell \prod_{H \in \calh} n^{-v(H)/2} \cdot \int_{n^\eps}^\infty \tau^{\abs{\calh}-1}g(\tau)\, d\tau.\]
But this is $n^{-\omega_{\lambda,\ell,\eps,\eta}(1)}$, so we are done.
\end{proof}

\subsection{Estimates on Boolean functions}
The proofs of \cref{lem:distrib,lem:decouple} will need a few results from the analysis of Boolean functions, which we state here. First, we will frequently use hypercontractivity, in the following two forms:
\begin{lem}[Moment hypercontractivity {\cite[Thm.~10.21]{O14}}] \label{lem:momhc}
Let $p \in (\lambda, 1-\lambda)$ and let $f$ be a polynomial of degree at most $D$ evaluated on independent $p$-biased random bits. Then for $q > 2$,
\[\sete[\abs{f}^q]^{1/q} \leq (\sqrt{q-1} \cdot \lambda^{1/q-1/2})^D \cdot \sete[\abs{f}^2]^{1/2}.\]
\end{lem}
\begin{lem}[Concentration hypercontractivity {\cite[Thm.~10.24]{O14}}] \label{lem:tbhc}
If $\lambda$, $D$, and nonzero $f$ are as above, then for any $t \geq (2e/\lambda)^{D/2}$, then
\[\setp[\abs{f} \geq t \sete[f^2]^{1/2}] \leq \lambda^D \exp \paren*{-\frac{D}{2e}\lambda t^{2/D}}.\]
\end{lem}
The only context in which we will use the latter result is through the following corollary.
\begin{cor} \label{cor:hyperc}
If $\lambda$, $D$, and nonzero $f$ are as above, then
\[\setp[\abs{f} \geq n^\eps \sete[f^2]^{1/2}] = n^{-\omega_{D,\lambda,\eps}(1)}.\]
\end{cor}
Finally, we have the following simple estimate on linear functions.
\begin{prop} \label{cor:linear}
Let $p \in (\lambda,1-\lambda)$ and let $\vx = (x_i)_{i \in I}$ be a vector of independent $p$-biased random bits. If $a \in \setr^I$ satisfies $\abs{a_i} \leq 1.2\pi$ for all $i \in I$, then
\[\abs{\sete[e^{i\va \cdot \vx}]} \leq e^{-\Omega_\lambda(\Abs{\va}_2^2)}.\]
\end{prop}
\begin{proof}
As the $x_i$ are independent, it suffices to show this result for $\abs{I} = 1$, i.e.~$\abs{(1-p) + pe^{ia}} \leq e^{-\Omega_\lambda(a^2)}$.
This is true since
\begin{multline*}
\abs{(1-p) + pe^{ia}}^2 = (1-p)^2+p^2+2p(1-p)\cos a \\ = 1 - 2p(1-p)(1-\cos a) \leq 1 - \frac{\lambda(1-\lambda)a^2}{5} \leq e^{-\lambda(1-\lambda)a^2/5},
\end{multline*}
where we have used the fact that $p(1-p) \geq \lambda(1-\lambda)$ and $1-\cos a \geq a^2/10$ for $\abs{a} \leq 1.2\pi$.
\end{proof}

Combining the previous results, we obtain the following result, which appears as
\cite[Theorem~3]{Ber18} and also occurs in \cite[Section~3.4.6]{SS22}, albeit both times with significantly different language. We include a proof for completeness.
\begin{lem} \label{lem:lu}
Let $p$, $\lambda$, $I$, $x_i$ be as in \cref{cor:linear}. For $i \in I$ let $\chi_i = (x_i-p)/\sqrt{p(1-p)}$ and for $S \subseteq I$ let $\chi_S = \prod_{i \in S} \chi_i$. For all $i \in I$, let $\delta_i$ be a real number. Let $\cals$ be a collection of subsets of $I$ of size at least $2$ and for all $S \in \cals$ let $\delta_S$ be a real number. Suppose $n$ is a positive integer and $\eps > 0$ such that
\begin{cond}
\item \label{cond:lu1}\label[scond]{scond:lu1}$\abs{I} \leq n^{1/\eps}$;
\item \label{cond:lu2}\label[scond]{scond:lu2}$\abs{S} \leq 1/\eps$ for all $S \in \cals$;
\item \label{cond:lu3}\label[scond]{scond:lu3}$\abs{\delta_i} \leq 1.2\pi \sqrt{p(1-p)}$ for all $i \in I$;
\item \label{cond:lu4}\label[scond]{scond:lu4}$\sum_{i \in I} \delta_i^2 \geq n^\eps$;
\item \label{cond:lu5}\label[scond]{scond:lu5}$\sum_{S \in \cals} \delta_S^2 \leq n^{-\eps}$.
\end{cond}
Then
\[\abs*{\sete\brac*{e^{i\paren*{\sum_{i \in I} \delta_i \chi_i + \sum_{S \in \cals} \delta_S \chi_S}}}} = n^{-\omega_{\lambda,\eps}(1)}.\]
\end{lem}
\begin{proof}
Let $L = \sum_{i \in I} \delta_i \chi_i$ and $U= \sum_{S\in\cals}\delta_S \chi_S$. Let $r$ be a positive integer. By Taylor approximation, we have
\[\abs*{e^{iU} - \sum_{j=0}^{r-1} \frac{(iU)^j}{j!}} \leq \frac{2\abs{U}^r}{r!},\]
so
\begin{equation} \label{eq:lubreakdown}
\abs*{\sete[e^{i(L+U)}]} = \abs*{\sete\brac*{\sum_{j=0}^{r-1} \frac{e^{iL}(iU)^j}{j!}} + \sete\brac*{e^{iL}\paren*{e^{iU} - \sum_{j=0}^{r-1} \frac{(iU)^j}{j!}}}} \leq \sum_{M \in \calm}\abs*{\sete\brac*{a_M M e^{iL}}} + \sete\brac*{\frac{2\abs{U}^r}{r!}},
\end{equation}
where $\calm$ is the set of monomials (in the $\chi_i$) occurring in $\sum_{j=0}^{r-1} \frac{(iU)^j}{j!}$ and
\[\sum_{M \in \calm} a_M M = \sum_{j=0}^{r-1} \frac{(iU)^j}{j!}.\]

Observe that by \cref{cond:lu2}, each $M \in \calm$ involves at most $r/\eps$ variables. For a given $M$, if we let $I'\subseteq I$ index the variables not in $M$, by \cref{cond:lu3,cond:lu4} we find that $\sum_{i \in I'} \delta_i^2 = \Omega_{\eps,r}(n^\eps)$. Applying \cref{cor:linear} (and \cref{cond:lu3}), we find that
\[\abs*{\sete\brac*{e^{\sum_{i \in I'} \delta_i \chi_i}}} = n^{-\omega_{\lambda,\eps,r}(1)},\]
so as a result, since $M$ is bounded in terms of $\lambda,\eps,r$, we have
\[\abs*{\sete\brac*{a_M M e^{iL}}} \leq \abs{a_M} n^{-\omega_{\lambda,\eps,r}(1)}.\]
To crudely bound $\sum_{M \in \calm} \abs{a_M}$, we note that by \cref{cond:lu5} we have $\abs{\delta_S} \leq 1$ for all $S \in \cals$. Also, by \cref{cond:lu1,cond:lu2} we find that $\abs{\cals} \leq n^{1/\eps^2}$. Therefore, if we imagine expanding out $\sum_{j=0}^{r-1} \frac{(iU)^j}{j!}$, we have at most $\sum_{j = 0}^{r-1} n^{j/\eps^2} \leq rn^{r/\eps^2}$ terms to collect, so $\sum_{M \in \calm} \abs{a_M} \leq rn^{r/\eps^2}$. As a result,
\begin{equation} \label{eq:mlterms} \sum_{M \in \calm}\abs*{\sete\brac*{a_M M e^{iL}}} \leq rn^{r/\eps^2}n^{-\omega_{\lambda,\eps,r}(1)} = n^{-\omega_{\lambda,\eps,r}(1)}.
\end{equation}

Finally, we observe that by \cref{lem:momhc}, we have
\begin{equation}\label{eq:uerror}
\sete\brac*{\frac{2\abs{U}^r}{r!}} \lsim_{\lambda,\eps,r} \sete[U^2]^{r/2} \leq n^{-r\eps/2}.
\end{equation}

Combining \labelcref{eq:lubreakdown}, \labelcref{eq:mlterms}, and \labelcref{eq:uerror}, we find that for all $r$,
\[\abs*{\sete\brac*{e^{i(L+U)}}} = n^{-\omega_{\lambda, \eps,r}(1)} + O_{\lambda,\eps,r}(n^{-r\eps/2}).\]
Since $r$ was arbitrary, we are done.
\end{proof}

\section{Distributional Result} \label{sec:distrib}
\cref{lem:distrib} is proven through a multivariate version of Stein's method of exchangeable pairs, which we now state.

For real matrices define the \vocab{Hilbert-Schmidt norm} $\Abs{A}_\HS = \sqrt{\tr(A^T A)}$ and the \vocab{operator norm} $\Abs{A}_\op = \max_{\abs{v} = 1} \Abs{Av}$. Also, for $f\in C^{k}(\setr^d)$ define
\[M_k(f) = \mathop{\smash{\sup}}_{\substack{x, u_1,\ldots,u_k \in\setr^d \\ \abs{u_i} = 1}} \abs{\partial_{u_1} \cdots \partial_{u_k} f(x)}.\]
Moreover, call two random variables $Y$ and $Y'$ \vocab{exchangeable} if $(Y,Y')$ and $(Y',Y)$ have the same distribution. 
\begin{thm}[\cite{M09}]\label{thm:exchangeable}
Let $(Y,Y')$ be an exchangeable pair of random vectors in $\setr^d$.  Suppose that
there is an invertible matrix $\Lambda$ such that $\sete\bracmid{Y'-Y}{Y} =-\Lambda Y$ and a 
random matrix $E$ such that $\sete\bracmid{E}{Y} =\sete\bracmid{(Y'-Y)(Y'-Y)^T-2\Lambda}{Y}$.
Then for $g\in C^3(\setr^d)$,
\[
\abs[\big]{\sete[g(Y)] - \sete[g(Z)]} \leq \Abs{\Lambda^{-1}}_\op \paren*{\frac{\sqrt{d}}{4}M_2(g) \sete\brac[\big]{\Abs{E}_\HS} + \frac{1}{9}M_3(g)\sete\brac[\big]{\abs{Y'-Y}^3}},
\]
where $Z$ is a standard multivariate normal on $\setr^d$.
\end{thm}
The following statement is our main distributional result.
\begin{thm} \label{thm:distrib}
Let $p\in (\lambda,1-\lambda)$ and $\calh = \set{H_1,\ldots,H_k}$ be set of nonisomorphic non-$\kone$ connected graphs. Letting $\sigma^2_{H} = \Var[\gamma_{H}]$ and $Z_1,\ldots,Z_k$ be a collection of independent standard normals, then for any 
$\psi \in C^3(\setr^k)$ we have
\[\abs*{\sete\brac[\big]{\psi(\gamma_{H_1}/\sigma_{H_{1}},\ldots, \gamma_{H_k}/\sigma_{H_{k}})}-\sete\brac[\big]{\psi(Z_1,\ldots,Z_k)}}\lsim_{\calh,\lambda} \frac{M_2(\psi)+M_3(\psi)}{n^{1/2}}.\]
\end{thm}
\begin{proof}
Aiming to apply \cref{thm:exchangeable}, our setup is as follows: let $X = (x_e)_{e \in\binom{[n]}{2}}$ be the indicator variables of the edges of $G(n,p)$. Sample $I$ from $\binom{[n]}{2}$ uniformly at random, and let $X'= (x'_e)_{e \in \binom{[n]}{2}}$ be $X$ but with $x_I$ resampled. Then, let
\[Y = \paren*{\frac{\gamma_{H_1}(X)}{\sigma_{H_1}},\ldots,\frac{\gamma_{H_k}(X)}{\sigma_{H_k}}}\quad\text{and}\quad Y' = \paren*{\frac{\gamma_{H_1}(X')}{\sigma_{H_1}},\ldots,\frac{\gamma_{H_k}(X')}{\sigma_{H_k}}}.\]
Since $X$ and $X'$ are exchangeable, so are $Y$ and $Y'$. Also, conditioning on $X$, every contribution to $\gamma_{H_i}(X)/\sigma_{H_i}$ has a $e(H_i)/\binom{n}{2}$ chance of having one of its edges resampled, which leads to a zero contribution in $\sete\bracmid{\gamma_{H_i}(X')/\sigma_{H_i}}{X,I}$. Thus, $\sete\bracmid{X'-X}{X} = -\Lambda X$
where $\Lambda = \diag(e(H_i)/\binom{n}{2})_{i \in [k]}$.
Finally, we define\[E = \sete\bracmid{(Y'-Y)(Y'-Y)^T - 2\Lambda}{X}.\] It thus remains to bound $\sete\brac[\big]{\abs{Y'-Y}^3},\sete\brac[\big]{\Abs{E}_\HS}\lsim_{\calh,\lambda}n^{-5/2}$.

First, we observe
\[\sete\brac[\big]{\abs{Y'-Y}^3} \leq \sete\brac[\big]{\abs{Y'-Y}^4}^{3/4} \leq \paren*{k\sum_{i\in[k]} \frac{\sete[(\gamma_{H_i}(X')-\gamma_{H_i}(X))^4]}{\sigma_{H_i}^4}}^{3/4}.\]
By symmetry, we may fix an $I$ in the above expectation. Then, $\gamma_{H_i}(X')-\gamma_{H_i}(X)$ consists of $O_\calh(n^{v(H_i)-2})$ terms and thus has variance $O_\calh(n^{v(H_i)-2})$. By \cref{lem:momhc}, it follows that $\sete[(\gamma_{H_i}(X')-\gamma_{H_i}(X))^4] \lsim_{\calh,\lambda} n^{2v(H_i)-4}$, and the bound follows since $\sigma_{H_i} \asymp_\calh n^{v(H_i)/2}$.

We now deal with $E$. Defining $\chi_e, \chi'_I$ similarly to the introduction and abbreviating $\prod_{e \in S} \chi_e$ as $\chi_S$, we have
\begin{align*}
\MoveEqLeft \sete\bracmid{(Y'_i - Y_i)(Y'_j - Y_j)}{X,I} \\ &= \frac{1}{\sigma_{H_i}\sigma_{H_j}}\sum_{\substack{H'_1 \cong H_i,\;H'_2 \cong H_j \\ I \in E(H'_1),E(H'_2)}}\sete\bracmid*{(\chi_{E(H'_1)\setminus I}\chi'_I - \chi_{E(H'_1)})(\chi_{E(H'_2)\setminus I}\chi'_I - \chi_{E(H'_2)})}{X,I} \\
&= \frac{1}{\sigma_{H_i}\sigma_{H_j}}\sum_{\substack{H'_1 \cong H_i,\;H'_2 \cong H_j \\ I \in E(H'_1),E(H'_2)}} (\chi_{E(H'_1)\setminus I}\chi_{E(H'_2)\setminus I} + \chi_{E(H'_1)}\chi_{E(H'_2)}),
\end{align*}
which implies that
\[E_{ij} = \frac{1}{\sigma_{H_i}\sigma_{H_j}\binom{n}{2}} \sum_{\substack{H'_1 \cong H_i,\;H'_2 \cong H_j \\ e \in E(H'_1),E(H'_2)}}(\chi_{E(H'_1)\setminus e}\chi_{E(H'_2)\setminus e} + \chi_{E(H'_1)}\chi_{E(H'_2)}) - \frac{2e(H_i)}{\binom{n}{2}} \delta_{ij},\]
where $\delta_{ij}$ is the Kronecker delta. At this point, it is easy to see that $\sete[E_{ij}] = 0$.\footnote{This also follows straightforwardly from the definition of $E$, exchangeability, and the fact that $\sete[YY^T]$ is the identity.} For future convenience, let $E'_{ij}$ be the sum in the above expression.

We may now bound
\[
\sete[\Abs{E}_\HS] \leq \sete[\Abs{E}_\HS^2]^{1/2} = \paren[\Big]{\sum_{i,j} \sete[E_{ij}^2]}^{1/2} \\ = \paren*{\sum_{i,j} \frac{\Var[E'_{ij}]}{\sigma_{H_i}^2\sigma_{H_j}^2 \binom{n}{2}^2}}^{1/2},
\]
which means that we need to show that $\Var[E'_{ij}] \lsim_{\calh,\lambda} n^{v(H_i) + v(H_j) - 1}$. To do this, we expand out
\[\Var[E'_{ij}] = \sum_{\substack{H'_1, H'_3 \cong H_i \\ H'_2,H'_4 \cong H_j}}\sum_{\substack{e_1 \in E(H'_1),E(H'_2) \\ e_2 \in E(H'_3),E(H'_4)}} \begin{multlined}[t] \Cov[\chi_{E(H'_1)\setminus e_1}\chi_{E(H'_2)\setminus e_1} + \chi_{E(H'_1)}\chi_{E(H'_2)}, \\ \chi_{E(H'_3)\setminus e_2}\chi_{E(H'_4)\setminus e_2} + \chi_{E(H'_3)}\chi_{E(H'_4)}]
\end{multlined}\]
and claim that all but $O_\calh(n^{v(H_i) + v(H_j) - 1})$ terms are zero, which finishes as each covariance is $O_{\calh,\lambda}(1)$.

If an edge appears exactly once in $E(H'_1), E(H'_2), E(H'_3), E(H'_4)$, the respective covariance vanishes. Thus, since $H_i, H_j$ have no isolated vertices, each vertex must appear in $V(H'_1), V(H'_2), V(H'_3), V(H'_4)$ at least twice. If a vertex appears more than twice, then the number of total vertices is at most $v(H_i) + v(H_j) - 1$, which corresponds to $O_\calh(n^{v(H_i) + v(H_j) - 1})$ terms. Thus we may assume that every vertex is covered exactly twice, and every edge at least twice.

$V(H'_1)$ must intersect $V(H'_2)$, as $E(H'_1) \cap E(H'_2) \neq \emptyset$. If $V(H'_1)$ also intersects $V(H'_3)$ or $V(H'_4)$, then since each vertex is covered exactly twice we may partition $V(H'_1)$ into multiple parts based on which other $V(H'_i)$ it is in. Then by the connectedness of $H_i$ there must exist an edge in $E(H'_1)$ joining two of $V(H'_2), V(H'_3), V(H'_4)$, which cannot be covered by any other edge set and thus leads to a zero covariance. Applying similar logic to $V(H'_2)$ implies that $V(H'_1) = V(H'_2)$, which in turn implies that $V(H'_3) = V(H'_4)$. However, in this case, the two variables in the covariance are independent, so we are done.
\end{proof}

\begin{proof}[Proof of \cref{lem:distrib}]
Let $\psi(z_H \colon H \in \calh) = e^{i\sum_H t_H F_H(\sigma_{H'} z_{H'} \colon H' \in \calh)}$ and let $\kappa\colon [0,\infty) \to [0,1]$ be a smooth decreasing cutoff function that is $1$ on $[0,1]$ and $0$ on $[2,\infty)$. Furthermore, let $\psi'(\vz) = \psi(\vz) \kappa(\Abs{\vz}_2/n^{\eps'})$, where $\eps' > 0$ is to be determined later.

Let $Y = (\gamma_H/\sigma_H)_{H \in \calh}$ and $Z$ be an $\calh$-indexed vector of independent standard normals. Then the quantities $\varphi^\calf_X(\vt)$ and $\varphi^\calh_Z(\vt)$ are precisely $\sete[\psi(Y)]$ and $\sete[\psi(Z)]$. By \cref{cor:hyperc} and direct computation, respectively, we have $\Abs{Y}_2, \Abs{Z}_2 \leq n^{\eps'}$ with probability $1 - n^{-\omega_{\lambda,\ell,\eps'}(1)}$, so since $\abs{\psi'(\vz) - \psi(\vz)} \leq 1$ always we have
\[\abs{\varphi^\calf_X(\vt) - \varphi^\calf_Z(\vt)} = \abs{\sete[\psi(Y) - \psi(Z)]} \leq \abs{\sete[\psi'(Y) - \psi'(Z)]} + n^{-\omega_{\lambda,\ell,\eps'}(1)}.\]

After applying \cref{thm:distrib} to $\psi'$, it suffices to show that $M_2(\psi'), M_3(\psi') \lsim_{\ell,\eps,\eta} n^{4\eps}$. To see this, observe that \cref{cond:ifs2} of \cref{def:ifs} implies that $F_H(\sigma_H z_H \colon H \in \calh)/n^{v(H)/2}$ is a polynomial in $\vz$ of bounded degree and coefficients bounded by $O_\ell(\eta)$. Thus, $\psi(\vz) = e^{if(\vz)}$ for some polynomial $f$ of bounded degree and coefficients bounded by $O_\ell(\eta \max_{H \in \calh} {\abs{t_H} n^{v(H)/2}}) \leq O_\ell(\eta n^\eps)$. Since taking $k$ partial derivatives of $e^{if}$ causes at most $k$ factors of $f$ and its derivatives to appear in front of the exponential, we conclude that $\abs{\partial_{u_1} \cdots \partial_{u_k} e^{if}} \lsim_{\ell,k} (1+\eta n^\eps)^k (1+\Abs{\vz}_2)^{O_{\ell,k}(1)}$ for unit vectors $u_1,\ldots,u_k$.

Since $\psi'$ is identically zero for $\Abs{\vz}_2 \geq 2n^{\eps'}$ its derivatives are certainly zero. Since $\kappa$ and its derivatives are bounded at other $\vz$, we conclude that $M_k(\psi') \lsim_{\ell,k} (1+\eta n^\eps)^k n^{O_{\ell,k}(\eps')}$. By choosing $\eps'$ to be sufficiently small, we therefore can get $M_2(\psi'), M_3(\psi') \lsim_{\ell,\eps,\eta} n^{4\eps}$, as desired.
\end{proof}

\section{Decoupling Methods} \label{sec:decouple}
In this section, we will prove \cref{lem:decouple}, which will complete the proof of \cref{thm:main}. Conceptually, this is quite simple: in \cref{subsec:dineq} we state a decoupling inequality, which we use to transform $\varphi^\calf_X(\vt)$ into an expression that can be bounded using \cref{lem:lu}. However, the precise details in this procedure are quite technical, and the bulk of the work consists of showing that regardless of the value of $\vt$, there is always exists a decoupling setup that will be effective. To accomplish this, in \cref{subsec:gendec} we develop a list of conditions (\cref{lem:decconditions}) required for decoupling to work in the case of random graphs. In \cref{subsec:setup1,subsec:setup2} we identify two families of decoupling setups that satisfy the conditions in \cref{lem:decconditions}, which are subsequently applied in \cref{subsec:decproof}.

\subsection{The decoupling inequality} \label{subsec:dineq}
Let $k$ be a nonnegative integer (which we will call the \vocab{decoupling parameter}). Following \cite{Ber18}, for a function $f \colon \Omega_0 \times \prod_{i=1}^k \Omega_k \to \setr$, where $\Omega_0, \Omega_1, \ldots, \Omega_k$ are arbitrary sets, define $\alpha(f) \colon \Omega_0 \times \prod_{i=1}^k \Omega_k^2 \to \setr$ given by
\[\alpha(f)(x, y_1^0, y_1^1,  \ldots, y_k^0, y_k^1) = \sum_{(i_1,\ldots,i_k) \in \set{0,1}^k} (-1)^{i_1 + \cdots + i_k} f(x, y_1^{i_1}, \ldots, y_k^{i_k}).\]
We then have the following result, which we prove for completeness.
\begin{lem}[\cite{Ber18}] \label{lem:dec}
Suppose $X, Y_1, Y_2, \ldots, Y_k$ are random variables on $\Omega_0, \Omega_1, \Omega_2, \ldots, \Omega_k$, respectively. If for $i \in [k]$ we let $Y'_i$ be an independent copy of $Y_i$, then
\[\abs[\big]{\sete_{X, Y}[e^{if(X,Y)}]}^{2^k} \leq \sete_{\vY}\brac*{\abs[\big]{\sete_X[e^{i\alpha(f)(X,\vY)}]}},\]
where we abbreviate $Y = (Y_1,\ldots,Y_k)$ and $\vY = (Y_1,Y'_1,\ldots,Y_k,Y'_k)$.
\end{lem}
\begin{proof}
We proceed with induction on $k$, with the $k=0$ case being trivial. Now assume the result is true for $k-1$ and let $\tilde Y = (Y_1,\ldots,Y_{k-1})$ and $\tilde \vY = (Y_1,Y'_1,\ldots,Y_{k-1},Y'_{k-1})$. By Cauchy-Schwarz we have
\[\abs[\big]{\sete_{X, Y}[e^{if(X,Y)}]}^{2^k} \leq \abs*{\sete_{X,\tilde Y}\brac*{\abs[\big]{\sete_{Y_k}[e^{if(X,\tilde Y,Y_k)}]}^2}}^{2^{k-1}} = \abs*{\sete_{X,\tilde Y}\brac*{\sete_{Y_k,Y'_k}[e^{i(f(X,\tilde Y,Y_k)-f(X,\tilde Y,Y'_k))}]}}^{2^{k-1}}.\]
Define the function
\[\tilde f((x, y_k,y'_k), y_1, \ldots, y_{k-1}) = f(x,y_1,\ldots,y_{k-1},y_k) - f(x,y_1,\ldots,y_{k-1},y'_k),\]
and observe that $\alpha(\tilde f)((X,Y_k,Y_k'),\tilde \vY) = \alpha(f)(X,\vY)$. Applying the inductive hypothesis to $\tilde f$, the quantity is therefore bounded by
\[\sete_{\tilde \vY}\brac*{\abs[\big]{\sete_{X,Y_k,Y_k'}[e^{i\alpha(f)(X,\vY)}]}} \leq \sete_{ \vY}\brac*{\abs[\big]{\sete_{X}[e^{i\alpha(f)(X,\vY)}]}},\]
as desired.
\end{proof}

\subsection{Decoupling on random graphs} \label{subsec:gendec}
Let $\calh$ be a set of (isomorphism classes of) NIV graphs on at most $\ell$ vertices. For $H \in \calh$, let $\Delta_H$ be real constants and define the statistic $f = \sum_{H\in \calh} \Delta_H \gamma_H$.

Take a nonnegative integer $k$ and a partition of $\binom{[n]}{2}$ into sets $A, B_1,\ldots,B_k,C$. Sample all the edges with probability $p$, letting $x_e$ and $\chi_e$ denote the indicator variable of an edge $e$ and its normalization, respectively. Furthermore, resample all edges in $\bigcup_i B_i$, defining $x'_e$ and $\chi'_e$ similarly. Let $X = (x_e)_{e\in A}$, $Y_i = (x_e)_{e \in B_i}$, $Y'_i = (x'_e)_{e \in B_i}$, and $W = (x_e)_{e \in C}$. As $f$ is naturally a function of $(X,W)$ and $Y_1,\ldots,Y_k$, we may apply \cref{lem:dec} to get
\begin{equation} \label{eq:maindec}
\abs{\sete_{G(n, p)}[e^{if}]}^{2^k} \leq \sete_{\vY} \brac*{\abs*{\sete_{X,W} e^{i\alpha(f)(X,W,\vY)}}} \leq \sete_{\vY,W} \brac*{\abs*{\sete_X e^{i\alpha(f)(X,W,\vY)}}}.
\end{equation}
For a set $S \subseteq A$, let $\delta_S$ be the function of $W, \vY$ such that $\alpha(f)(X,W,\vY) = \sum_S \delta_S \chi_S$. (Abbreviate $\delta_e = \delta_{\set{e}}$.) For a set of edges $T$ disjoint from $A$, define
\[\zeta_T = \smashoperator{\prod_{e \in T \cap C}} \chi_e \cdot  \prod_{i \in [k]} \paren*{\smashoperator[r]{\prod_{e \in T \cap B_i}} \chi_e - \smashoperator[r]{\prod_{e \in T \cap B_i}} \chi_e'}.\]
Call a set of edges, and by extension a subgraph, \emph{rainbow} if it intersects $B_i$ for all $i$. For a graph $H$ and a subset $S \subseteq A$, let $N_{H, S}$ be the number of rainbow copies $H'$ of $H$ such that $E(H') \cap A = S$.

We make the following easy observations:
\begin{enumerate}
    \item Each $\zeta_T$ is a polynomial in the $x_e$ of degree at most $\abs{T}$.
    \item \label{item:dec2} If $T$ is not rainbow, $\zeta_T = 0$.
    \item \label{item:dec3} For rainbow $T$ and $T'$, we have $\sete[\zeta_T\zeta_{T'}] = 2^k\cdot  \one_{T=T'}$.
    \item If $T_1,T_2,T_3,T_4$ are such that there exists an edge contained within exactly one $T_i$, then $\sete[\zeta_{T_1} \zeta_{T_2} \zeta_{T_3} \zeta_{T_4}] = 0$.
    \item \label{item:dec5} We have
    \[\delta_S = \sum_{H \in \calh} \Delta_{H} \;\smashoperator{\sum_{\substack{H'\cong H \\ E(H') \cap A = S}}} \zeta_{E(H') \setminus A}.\]
    \item As a consequence of \labelcref{item:dec2,item:dec3,item:dec5}, we have $\sete[\delta_S^2] = 2^k \sum_{H \in \calh} \Delta_H^2 N_{H,S}$. (Note that this deduction requires all $H\in\calh$ to have no isolated vertices, since otherwise two different graphs could have the same edge sets.)
\end{enumerate}

We may now express our conditions for successful decoupling in terms of the following lemma:
\begin{lem} \label{lem:decconditions}
Suppose there is a nonempty subset $\calh' \subseteq \calh$ and some $\eps > 0$ such that
\begin{cond}
    \item \label{cond:g1} \label[scond]{scond:g1} We have one of the following:
    \begin{cond}
        \item \label{cond:g1a} \label[scond]{scond:g1a} $\Delta^2_H N_{H, e} \leq n^{-\eps}$ for all $H \in \calh'$ and $e \in A$;
        \item \label{cond:g1b} \label[scond]{scond:g1b} For $e \in A$ we always have
        \[\abs*{\sum_{H \in \calh'} \Delta_H \;\smashoperator{\sum_{\substack{H'\cong H \\ E(H') \cap A = \set{e}}}} \zeta_{E(H') \setminus A}} \leq 1.1\pi\sqrt{p(1-p)}.\]
    \end{cond}
    \item \label{cond:g2} \label[scond]{scond:g2} $\sum_{H \in \calh'} \sum_{e \in A} \Delta^2_H N_{H, e} \geq n^\eps$;
    \item \label{cond:g3} \label[scond]{scond:g3} $\sum_{e \in A}  \Delta^2_H N_{H, e} \leq n^{-\eps}$ for all $H \in \calh \setminus \calh'$;
    \item \label{cond:g4} \label[scond]{scond:g4} $\sum_{\abs{S} \geq 2} \Delta^2_H N_{H, S} \leq n^{-\eps}$ for all $H \in \calh$;
    \item \label{cond:g5} \label[scond]{scond:g5} We have
    \[\sum_{(H_1,H_2,H_3,H_4) \in Q} \Delta_{H_1}\Delta_{H_2}\Delta_{H_3}\Delta_{H_4} \leq n^{-\eps} \paren*{\sum_{H \in \calh'} \sum_{e \in A} \Delta^2_H N_{H, e}}^2,\]
    where $Q$ is the set of quadruples $(H_1,H_2,H_3,H_4)$ of subgraphs such that
    \begin{cond}
        \item \label{cond:g5a} \label[scond]{scond:g5a} $H_1, H_2, H_3, H_4$ are isomorphic to elements of $\calh'$ and are rainbow ($\Delta_{H_i}$ denotes $\Delta_H$ where $H\in\calh'$ is isomorphic to $H_i$);
        \item \label{cond:g5b} \label[scond]{scond:g5b} $E(H_1) \cap A = E(H_2) \cap A = \set{e}$ for some $e \in A$;
        \item \label{cond:g5c} \label[scond]{scond:g5c} $E(H_3) \cap A = E(H_4) \cap A = \set{e'}$ for some $e' \in A$;
        \item \label{cond:g5d} \label[scond]{scond:g5d} $(E(H_1) \cup E(H_2)) \cap (E(H_3) \cup E(H_4))$ is nonempty;
        \item \label{cond:g5e} \label[scond]{scond:g5e} No edge lies in exactly one of the $E(H_i)$.
    \end{cond}
\end{cond}
Then $\sete_{G(n,p)}[e^{if}] = n^{-\omega_{\lambda,\ell,\eps}(1)}$.
\end{lem}
\begin{proof}
Note that we must have $k < \binom{\ell}{2}$, since otherwise we would always have $N_{H, e} = 0$ and \cref{cond:g2} would be impossible to satisfy. Thus, by \labelcref{eq:maindec}, it suffices to show that
\[\sete_{\vY,W} \brac*{\abs*{\sete_X e^{i\alpha(f)(X,W,\vY)}}} = n^{-\omega_{\lambda,\ell,\eps}(1)}.\]
Since $e^{i\alpha(f)(X,W,\vY)}$ is bounded, we will show that, with the $\delta_e$ and $\delta_S$ defined as above, the conditions of \cref{lem:lu} (with $I = A$) hold with probability $1 - n^{-\omega_{\lambda,\ell,\eps}(1)}$ (call this \vocab{high probability} for the remainder of the proof) for some suitable $\eps'(\lambda,\ell,\eps)$ (it is clear that the $\delta_\emptyset$ term can be ignored). This will finish the proof. To prevent ambiguity, we will let $\eps'$ refer to the $\eps$ in \cref{lem:lu} and denote the conditions of \cref{lem:lu} with square brackets.

\Cref{scond:lu1,scond:lu2} are obvious provided that $\eps' \leq \ell^{-2}$. Moreover, since there are only polynomially many $S$, by \cref{cor:hyperc} we know that with high probability $\delta_S^2 \leq n^{\eps/2} \sete[\delta_S^2]$ for all $S$. In this case, by \cref{cond:g4} we have 
\[\sum_{\abs{S} \geq 2} \delta_S^2 \leq 2^k n^{\eps/2} \sum_{H \in \calh} \sum_{\abs{S}\geq 2} \Delta_H^2 N_{H,S} = O_\ell(n^{-\eps/2}).\]
Thus \cref{scond:lu5} is satisfied provided that $\eps' < \eps/2$.

To deal with \cref{scond:lu3,scond:lu4}, we decompose $\delta_e = \delta'_e + r_e$, where
\[\delta'_e = \sum_{H \in \calh'} \Delta_H \;\smashoperator{\sum_{\substack{H'\cong H \\ E(H') \cap A = \set{e}}}} \zeta_{E(H') \setminus A} \quad\text{and}\quad r_e = \sum_{H \in \calh\setminus\calh'} \Delta_H \;\smashoperator{\sum_{\substack{H'\cong H \\ E(H') \cap A = \set{e}}}} \zeta_{E(H') \setminus A}.\]
By \cref{cond:g3} and \cref{cor:hyperc}, we have $\abs{r_e} \leq n^{-\eps/4}$ for all $e$ with high probability. If \cref{cond:g1a} holds, then we also have $\abs{\delta'_e} \leq n^{-\eps/4}$ for all $e$ with high probability, which yields \cref{scond:lu3}. \Cref{cond:g1b} simply states that $\abs{\delta'_e} \leq 1.1\pi\sqrt{p(1-p)}$ always, so combining this with our bound on $r_e$ also yields \cref{scond:lu3}.

It remains to show \cref{scond:lu4}, which is the most involved part of the proof. First of all, by the triangle inequality on $\setr^A$ we have
\[\sum_{e \in A} \delta_e^2 \geq \paren*{\sqrt{\sum_{e \in A} (\delta'_e)^2} - \sqrt{\sum_{e \in A} r_e^2}}^2.\]
By \cref{cor:hyperc}, we have $r_e^2 \leq n^{\eps/2} \sete[r_e^2]$ for all $e$ with high probability, and in this case by \cref{cond:g3} it is true that $\sqrt{\sum_{e \in A} \sete[r_e^2]} \lsim_\ell n^{-\eps/4}$. Thus it suffices to show that $\sum_{e \in A}(\delta'_e)^2 \geq n^{\eps'}$ for some $\eps'>0$ with high probability.

To show this, we apply \cref{cor:hyperc} on $\sum_{e \in A}((\delta'_e)^2 - \sete[(\delta'_e)^2])$; in light of \cref{cond:g2}, for this to work, it suffices to show that
\begin{equation}\label{eq:sqdconcen}
\Var\brac*{\sum_{e \in A}(\delta'_e)^2} \leq n^{-\eps/2}\paren*{\sum_{e\in A}\sete[(\delta'_e)^2]}^2
\end{equation}
for large $n$.

At this point, we expand
\[
\Var\brac*{\sum_{e \in A}(\delta'_e)^2} =\sum_{H_1,H_2,H_3,H_4} \Delta_{H_1}\Delta_{H_2}\Delta_{H_3}\Delta_{H_4} \Cov[\zeta_{E(H_1) \setminus A}\zeta_{E(H_2) \setminus A},\zeta_{E(H_3) \setminus A}\zeta_{E(H_4) \setminus A}]\]
where we sum over $H_1,H_2,H_3,H_4$ satisfying \cref{cond:g5a,cond:g5b,cond:g5c}. If such a quadruple is not in $Q$, it either violates \labelcref{cond:g5e}, in which case
\[\sete[ \zeta_{E(H_1) \setminus A}\zeta_{E(H_2) \setminus A}\zeta_{E(H_3) \setminus A}\zeta_{E(H_4) \setminus A}] = \sete[ \zeta_{E(H_1) \setminus A}\zeta_{E(H_2) \setminus A}]\sete[\zeta_{E(H_3) \setminus A}\zeta_{E(H_4) \setminus A}] = 0,\]
or it violates \labelcref{cond:g5d}, in which case $\zeta_{E(H_1) \setminus A}\zeta_{E(H_2) \setminus A}$ and $\zeta_{E(H_3) \setminus A}\zeta_{E(H_4) \setminus A}$ are independent and the entire term is $0$. Since the $\zeta_T$ are bounded in terms of $\lambda$ and $\ell$, we conclude that
\[\sete\brac*{\paren*{\sum_{e \in A}((\delta'_e)^2 - \sete[(\delta'_e)^2])}^2} \lsim_{\lambda,\ell} \sum_{(H_1,H_2,H_3,H_4) \in Q} \Delta_{H_1}\Delta_{H_2}\Delta_{H_3}\Delta_{H_4}.\]
So \cref{cond:g5} is exactly what we need to conclude \labelcref{eq:sqdconcen} for large $n$. This concludes the proof.
\end{proof}

\subsection{Decoupling setup I: midscale uniformity} \label{subsec:setup1}
\begin{lem} \label{lem:setup1}
Let $0 \leq \alpha,\beta \leq 1-\eps$ be real parameters and let $k$ be a nonnegative integer. Then $\abs{\sete[e^{if}]} = n^{-\omega_{\lambda,\ell,\eps}(1)}$, provided that
\begin{cond}
	\item \label{cond:m1} For all $H$ with $k+2$ vertices, $n^{k+2}\Delta_H^2 \leq n^{2+k\beta-\eps}$;
    \item \label{cond:m2} There is some connected $H$ with $k + 2$ vertices such that $n^{2\alpha+k\beta+\eps} \leq n^{k+2}\Delta_H^2$;
    \item \label{cond:m3} For every $H$ with greater than $k+2$ vertices 
    we have $n^{v(H)}\Delta_H^2 \leq n^{2\alpha+k\beta - \eps}$.
\end{cond}
\end{lem}
\begin{proof}
We let $k$ be the decoupling parameter. Construct $k+1$ disjoint vertex sets $V_0,V_1,\ldots,V_k$ such that $\abs{V_0} = \floor{\frac{1}{k+1}n^{1-\alpha}}$ and $V_i = \floor{\frac{1}{k+1} n^{1-\beta}}$ for $i \in [k]$. Let $A$ be edges within $V_0$, $B_i$ be edges with one endpoint in $V_i$ and one in $V_j$ for some $0 \leq j \leq i$, and $C$ be all other edges. It suffices to show that the conditions of \cref{lem:decconditions} hold for large $n$ and with $\eps$ replaced with $\eps/2$. Similar to before, we denote the conditions of \cref{lem:decconditions} with square brackets to prevent ambiguity.

Observe that any rainbow graph must have at least one vertex in $V_i$ for all $i \in [k]$. Thus
\begin{equation} \label{eq:nhs1}
N_{H,S} \lsim_\ell n^{v(H) -  v(S) - k\beta},
\end{equation} where $v(S)$ denotes the number of vertices adjacent to at least one edge in $S$. Furthermore, we claim that this is tight if $H$ is connected with $k+2$ vertices and $S$ is an edge $e$. To see this, pick any edge $e'$ in $H$ and label the vertices in $H$ not adjacent to $e'$ with the numbers $1, \ldots, k$ such that every vertex is adjacent to a vertex of $e'$ or a vertex with a smaller label. Then, sending $e'$ to $e$ and the vertex labeled $i$ to an arbitrary vertex in $V_i$ yields a rainbow copy of $H$; moreover, this can be done in $\Omega_{\ell,\eps}(n^{k(1-\beta)})$ ways. 

It is now simple to check \cref{scond:g1,scond:g2,scond:g3,scond:g4}. We let $\calh'$ consist of graphs in $\calh$ with $k+2$ vertices. Then by \cref{cond:m1} and \labelcref{eq:nhs1}, we find that for all $H \in \calh'$ and $e \in A$ we have $\Delta_H^2 N_{H,e} \lsim_{\ell,\eps} n^{-\eps}$, implying \cref{scond:g1a}. By \cref{cond:m2}, there is some $H \in \calh'$ with $\Delta_H^2 N_{H,e} \gsim_{\ell,\eps} n^{-2(1-\alpha) + \eps}$ for all $e \in A$; summing over the $\Theta_{\ell,\eps}(n^{2(1-\alpha)})$ elements of $A$ yields \cref{scond:g2}.

For all $H \in \calh \setminus \calh'$, it is either the case that $H$ has fewer than $k+2$ vertices, in which case $N_{H,S} = 0$ for all $S \subseteq A$, or that \cref{cond:m3} applies, which implies that  $\Delta_H^2 N_{H,S} \lsim_{\ell,\eps} n^{2\alpha - v(S)-\eps}$. For each isomorphism class of the graph associated with $S$, there are $\Theta_{\ell,\eps}(n^{(1-\alpha) v(S)})$ corresponding subsets of $A$, so summing over the $O_\ell(1)$ such isomorphism classes yields that $\sum_{\abs{S} \geq 1} \Delta_H^2 N_{H,S} \lsim_{\ell,\eps} n^{-\eps}$, proving both \cref{scond:g3,scond:g4} for $H \in \calh \setminus \calh'$. (Here we have used $v(S) \geq 2$.) \Cref{scond:g4} also holds for $H \in \calh'$ as well since for such $H$ we have $N_{H,S} = 0$ whenever $v(S) \geq 3$, which is the case for all $S$ with $\abs{S} \geq 2$.

Finally, to show \cref{scond:g5}, we claim that $\abs{Q} \lsim_{\ell,\eps} n^{4(1-\alpha)+2k(1-\beta)-\eps}$, which will suffice as we will then have
\[\sum_{(H_1,H_2,H_3,H_4) \in Q} \Delta_{H_1}\Delta_{H_2}\Delta_{H_3}\Delta_{H_4} \lsim_{\ell,\eps} n^{4(1-\alpha)+2k(1-\beta)-\eps} \max_{H \in \calh'} \Delta_H^4,\]
while the tightness of our bound on $N_{H,e}$ for $H \in \calh'$ gives
\[\paren*{\sum_{e \in A} \sum_{H \in \calh'}  \Delta^2_H N_{H, e}}^2 \gsim_{\ell,\eps} n^{4(1-\alpha)+2k(1-\beta)} \max_{H \in \calh'} \Delta_H^4.\]
To see this, consider some $(H_1,H_2,H_3,H_4) \in Q$ and let $V = \bigcup_{i \in [4]} V(H_i)$. We know that each $H_i$ must have two vertices in $V_0$, so $\abs{V \cap V_0} \leq 4$. Also, each $H_i$ must have exactly one vertex in each $V_j$ (for $j \in [k]$), and moreover, if some vertex $v \in V_j$ is a vertex of exactly one $H_i$, then the edges in $E(H_i) \cap B_j$ cannot be in $H_{i'}$ for any $i' \neq i$. Therefore $\abs{V \cap V_j} \leq 2$.

Furthermore, suppose for the sake of contradiction that equality holds in all of these bounds, i.e.~that $\abs{V \cap V_0} = 4$ and $\abs{V \cap V_j} = 2$ for all $j \in [k]$. By \cref{scond:g5d}, the set $V^* = (V(H_1) \cup V(H_2)) \cap (V(H_3) \cup V(H_4))$ is nonempty, and let $j$ be minimal such that $V^* \cap V_j$ is nonempty. We cannot have $j = 0$, as it would contradict \cref{scond:g5b,scond:g5c} and our assumption that $\abs{V \cap V_0} = 4$. Thus, we may assume, after possibly swapping $H_3$ and $H_4$, that $V(H_1) \cap V_j = V(H_3) \cap V_j$ consists of one vertex that is not the element of $V(H_2) \cap V_j$. However, since $H_1$ is rainbow, it contains an edge in $B_j$, and the only way for that edge to be covered by another $H_i$ is if it is in $H_3$. But this implies that $H_1$ and $H_3$ share a vertex in $V_{j'}$ for some smaller $j'$, contradicting the minimality of $j$. We conclude that we cannot have $\abs{V \cap V_0} = 4$ and $\abs{V \cap V_j} = 2$ for all $j \in [k]$.

Consequently, the number of choices for $V$ is bounded above by
\[\abs{V_0}^3 \prod_{i \in [k]} \abs{V_i}^2 + \sum_{i \in [k]} \paren*{\abs{V_0}^4 \abs{V_i} \prod_{j \in [k] \setminus \set{i}} \abs{V_j}^2} \lsim_{\ell} n^{4(1-\alpha)+2k(1-\beta)} (n^{\alpha-1} + n^{\beta - 1}) \lsim n^{4(1-\alpha)+2k(1-\beta)-\eps}.\]
Since $V$ is of bounded size, there are $O_\ell(1)$ many elements of $Q$ with a given $V$, so we have proved the desired bound.
\end{proof}

\subsection{Decoupling setup II: hyperlocal uniformity} \label{subsec:setup2}
\begin{lem} \label{lem:setup2}
Let $H\in \calh$ be connected and let $\eps \leq \alpha \leq 1$ be real. We get $\abs{\sete[e^{if}]} = n^{-\omega_{\lambda,\ell,\eps}(1)}$ as long as
\begin{cond}
    \item \label{cond:h1} $n^{-\alpha+\eps} \leq \Delta_H^2 \leq 1.1\pi^2 (p(1-p))^{e(H)}$;
    \item \label{cond:h2} For any graph $H' \succ H$, we have $n^{v(H')} \Delta_{H'}^2 < n^{v(H)- \alpha - \eps}$.
\end{cond}
\end{lem}
\begin{proof}
We let $k = e(H) - 1$ and arbitrarily label the edges of $H$ as $e_0,\ldots,e_k$. Consider $\floor{n^{\alpha}/v(H)}$ vertex-disjoint copies of $H$ (call these \vocab{standard copies}) and let $A$ be the set of all copies of $e_0$, while for all $i \in [k]$ let $B_i$ be the set of all copies of $e_i$. Place all other edges in $C$. We will let $\calh' = \set{H}$.

Call a rainbow graph \vocab{super-rainbow} if it has nontrivial intersection with $A$. Observe that every super-rainbow graph $H'$ must satisfy $H' \succeq H$ (by ``gluing together'' the standard copies). Moreover, we claim that a super-rainbow graph $H'$ containing edges from $a$ standard copies of $H$ has at least $v(H) + a - 1$ vertices. This can be proven by induction; the base case $a = 1$ is clear, and if $a \geq 2$ we may first assume without loss of generality that $E(H') \cap C = \emptyset$. Then, by the connectedness of $H$, there are two standard copies of $H$ and two corresponding vertices in those copies that are both in $V(H')$. After moving all edges in one of these copies to the other, which necessarily decreases the vertex count by at least $1$, we may apply the inductive hypothesis. One immediate consequence of this claim are that the standard copies are the only super-rainbow copies of $H$. A less-immediate consequence is that for any $H' \in \calh$, the number of super-rainbow copies of $H'$ is $O_\ell(n^{v(H')-v(H)+\alpha})$. To see this, suppose such a copy has edges in $a$ standard copies, which are incident to $b \geq v(H) + a - 1$ vertices. Then, each copy can be produced by choosing edges in the standard copies, which can be done in $O_\ell(n^{\alpha a})$ ways, identifying those chosen edges with a subgraph of $H'$, which can be done in $O_\ell(1)$ ways, and choosing the rest of the vertices of $H'$, which can be done in $O_\ell(n^{v(H') - b})$ ways, for $O_\ell(n^{v(H') - b + a +(\alpha-1) a}) \leq O_\ell(n^{v(H') - v(H) +\alpha})$ ways in total. The bound follows from summing over the $O_\ell(1)$ possible values of $a$ and $b$.

We are now ready to check the conditions of \cref{lem:decconditions} with $\eps$ replaced with $\eps/2$, which we again denote using square brackets. Observe that if $H'$ is a standard copy of $H$, then $\zeta_{E(H') \setminus A}$ is a product of $k$ terms that are bounded by $1/\sqrt{p(1-p)}$; thus since $N_{H,e} = 1$ for all $e \in A$ by \cref{cond:h1} we find that we always have
\[\abs*{\Delta_H \;\smashoperator{\sum_{\substack{H'\cong H \\ E(H') \cap A = \set{e}}}} \zeta_{E(H') \setminus A}} \leq \sqrt{1.1}\,\pi \frac{(p(1-p))^{e(H)/2}}{(p(1-p))^{k/2}} = \sqrt{1.1}\,\pi\sqrt{p(1-p)},\]
showing \cref{scond:g1b}. Moreover, \cref{cond:h1} also yields that $\Delta_H^2 \geq n^{-\alpha+\eps}$, which is enough to show \cref{scond:g2}. \Cref{scond:g4} also holds for $H$ as $N_{H,S} = 0$ for $\abs{S} \geq 2$.

If $H' \neq H$, then either $H' \not\succ H$, in which case $N_{H,S} = 0$ for all $\abs{S} \geq 1$ and \cref{scond:g3,scond:g4} are trivial, or \cref{cond:h2} applies. In this case, by our above discussion we have that $\sum_{\abs{S} \geq 1} N_{H,S} \lsim_\ell n^{v(H') - v(H) + \alpha}$, so
\[\sum_{\abs{S} \geq 1} \Delta_H^2 N_{H,S} \lsim_\ell n^{-\eps}.\]
This implies \cref{scond:g3,scond:g4}.

Finally, to show \cref{scond:g5}, we claim that $Q$ simply consists of quadruples of four equal standard copies of $H$. Indeed, \cref{scond:g5a,scond:g5b,scond:g5c} imply that any $(H_1,H_2,H_3,H_4) \in Q$ must consist of four super-rainbow copies of $H$ and thus must be standard copies, and \cref{scond:g5b,scond:g5c} imply that $H_1 = H_2$ and $H_3 = H_4$. \Cref{scond:g5d} ensures that $H_1 = H_2 = H_3 = H_4$. Therefore,
\[\sum_{(H_1,H_2,H_3,H_4) \in Q} \Delta_{H_1}\Delta_{H_2}\Delta_{H_3}\Delta_{H_4} \leq n^{\alpha}\Delta_H^4,\]
while
\[\paren*{\sum_{e \in A} \sum_{H \in \calh'}  \Delta^2_H N_{H, e}}^2 \gsim_{\ell,\eps} n^{2\alpha}\Delta_H^4,\]
which is enough to show \cref{scond:g5}.
\end{proof}

\subsection{Proof of \texorpdfstring{\cref{lem:decouple}}{Lemma \ref{lem:decouple}}} \label{subsec:decproof}
Let $\calh_r$ be the set of $r$-vertex graphs in $\calh$ and let $\calh'$ be the set of (possibly disconnected) NIV graphs $H'$ such that $H' \preceq H$ for some $H \in \calh$. By \cref{cond:ifs3} of \cref{def:ifs}, we can write $F_H = \sum_{H' \in \calh'} c_{H,H'} \gamma_{H'}$ where $c_{H,H} = (p(1-p))^{e(H)/2}$ and $\abs{c_{H,H'}} \leq \eta n^{(v(H)-v(H'))/2}$; in particular, we have $\sum_{H \in \calh} t_H F_H = \sum_{H' \in \calh'} \Delta_{H'} \gamma_{H'}$, where we define $\Delta_{H'} = \sum_{H \in \calh} c_{H,H'} t_H$. Recall that $c_{H,H'} \neq 0$ implies $H\succeq H'$.

For $2 \leq r \leq \ell$, let $L_r = \max_{H\in\calh_r} n^{r/2} \abs{t_H}$. By hypothesis, there is some $r$ with $L_r \geq n^{\eps}$; therefore, we may pick some $m$ such that $L_m \geq n^{\eps/\ell} \max(1, \max_{r > m} L_r)$.
We now split into cases.

Suppose $L^2_m \leq n^{m-\eps}$. From the bounds on $c_{H,H'}$, the contribution to $n^{v(H')/2}\Delta_{H'}$ from $\sum_{H \in \calh_r} t_H F_H$ is $O_{\ell,\eta}(L_r)$. It follows that for $H \in \calh'$ with more than $m$ vertices we have $n^{v(H)} \Delta_H^2 \lsim_{\ell,\eta} n^{-2\eps/\ell}L_m^2$. Moreover, for $m$-vertex $H \in \calh'$, we have $n^m \Delta_H^2 \lsim_{\ell,\eta} L_m^2$ and $n^{m/2}\abs{\Delta_H - \Delta'_H} \lsim_{\ell,\eta} n^{-\eps/\ell} L_m$,
where we define $\Delta'_H = \sum_{H' \in \calh_m} c_{H',H} t_{H'}$.

Now, let $\beta$ be such that $L_m^2 = n^{m\beta + \eps/\ell}$; note that
\[0 \leq \beta \leq 1-\frac{\eps+\eps/\ell}{m} \leq 1-\frac{\eps}{\ell}.\] We then apply \cref{lem:setup1} with $\alpha$ and $\beta$ equal to this choice of $\beta$, $k = m-2$, and with $\eps$ replaced with $\eps/(2\ell)$. \Cref{cond:m3} follows from our previous discussion, whereas to show \cref{cond:m1} we need to verify for some positive $\Omega_{\ell,\eps}(1)$ that
\[m\beta + \eps/\ell + \Omega_{\ell,\eps}(1) \leq 2 + (m-2)\beta - \frac{\eps}{2\ell} \iff 2\beta + \Omega_{\ell,\eps}(1) \leq 2 - \frac{3\eps}{2\ell},\]
which is true. Finally, to show \cref{cond:m2} we need to prove that there exists some $H \in \calh_m$ with $n^{m/2} \abs{\Delta_H} \geq n^{m\beta /2+\eps/(4\ell)} = n^{-\eps/(4\ell)} L_m$. Suppose not; then by our bound on $\abs{\Delta_H - \Delta'_H}$ we must then have $n^{m/2} \abs{\Delta'_H} \lsim_{\ell,\eta} n^{-\eps/(4\ell)} L_m$. Let $M = (c_{H',H})_{H,H' \in \calh_m}$ be a matrix with rows and columns indexed by $\calh_m$; by construction, $(\Delta'_H)_{H \in \calh_m} = M\cdot (t_H)_{H \in \calh_m}$. Observe that under the $\preceq$ order, $M$ is upper triangular, has entries bounded by $\eta$, and has diagonal entries that are $\Theta_{\lambda,\ell}(1)$. As a result, $M\inv$ has entries that are $O_{\lambda,\ell,\eta}(1)$, which implies that $n^{m/2} \abs{t_H} \lsim_{\lambda,\ell,\eta} n^{-\eps/(4\ell)} L_m$ for all $H \in \calh_m$, contradicting the definition of $L_m$.

If $n^{m-\eps} < L_m^2 \leq \pi^2 n^m$, further let $\calh_{m,s}$ be the set of graphs in $\calh$ with $m$ vertices and $s$ edges and define $R_s = \max_{H \in \calh_{m,s}} n^{m/2} \abs{t_H}$. Since $L_m \geq n^{\eps/\ell} \max(1, \max_{r > m} L_r)$, we may pick some $u$ such that $R_u \geq n^{\eps/\ell^3} \max(L_m/n^{\eps/\ell}, \max_{s>u} R_u)$. By similar logic to the first case we find that for all $H \in \calh'$ with greater than $m$ vertices or with $m$ vertices and greater than $u$ edges, we have $n^{v(H)}\Delta_H^2 \lsim_{\ell,\eta} n^{-2\eps/\ell^3} R_u^2$. Moreover, since $H \prec H'$ implies $v(H) < v(H')$ or $e(H) < e(H')$, there exists some $H \in \calh_{m,u}$ with $n^{v(H)} \Delta_H^2 = (1 + O_{\ell,\eta}(n^{-\eps/\ell^3}))(p(1-p))^uR_u^2$. Choose $\alpha$ such that $R_u^2 = n^{m-\alpha+\eps/\ell^3}$, and note that
\[\frac{\eps}{\ell^3} - \frac{\log \pi^2}{\log n} \leq \alpha \leq \eps + \frac{2\eps}{\ell} - \frac{\eps}{\ell^3} \leq 2\eps.\]
The conditions of \cref{lem:setup2} can straightforwardly be shown to hold for this choice of $H$ and $\alpha$ and with $\eps/(2\ell^3)$ replacing $\eps$, as long as $\eps \leq \frac{1}{2}$ and $n$ is sufficiently large, both of which we may assume. This completes the proof.

\section{Application to Proportional Graphs} \label{sec:prop}
Qualitatively speaking, \cref{thm:main} says that given a list of desired connected subgraph counts for an $n$-vertex graph that are close to their expectations in $G(n, p)$, the only obstruction to the existence of a graph with those subgraph counts for large $n$ is integrality. In this section, we apply this idea to enumerate and show the existence of $p$-proportional graphs and related concepts. We begin with the following definition of Janson.
\begin{defn}[{\cite[\pnfmt{64}]{Jan94}}] \label{def:phhat}
Given a nonempty set of graphs $\calh$ and some $0 < p < 1$, a graph $G$ is \vocab{$\suphat{(p,\calh)}$-proportional} if $v(G) \geq \max\set{v(H) \colon H \in \calh}$ and $\gamma_H(G) = 0$ for all $H \in \calh$.
\end{defn}
Letting $\calc_k$ be the set of (isomorphism classes of) connected $k$-vertex graphs, we additionally call $G$ \vocab{$p$-proportional} if $G$ is $\suphat{(p, \calc_2 \cup \calc_3)}$-proportional, \vocab{$p$-superproportional} if $G$ is $\suphat{(p, \calc_2 \cup \calc_3 \cup \calc_4)}$-proportional, and \vocab{$p$-hyperproportional} if $G$ is $\suphat{(p, \calc_3 \cup \calc_4 \cup \calc_5 \cup \set{2K_2})}$-proportional. These definitions are motivated by the fact that they appear as possible degenerate cases in the distribution of the number of induced copies of $G$ in the $G(n,p)$ and $G(n,m)$ models (for more details see \cite[Sec.~10]{Jan94}).
\begin{rmk}
The terms ``$p$-proportional'' and ``$p$-superproportional'' are have appeared previously in the literature \cite{JK91,Jan94}, while the term ``$p$-hyperproportional'' is new. We caution the reader that although $p$-superproportionality is a stronger condition than $p$-proportionality, it is impossible for a graph to be $\suphat{(p, \set{K_2, P_2, 2K_2})}$-proportional (see \cite{Jan95} or the discussion below), implying that $p$-hyperproportionality and $p$-proportionality are in fact disjoint conditions.
\end{rmk}
Since every $p$-proportional graph $G$ must have $e(G) = p\binom{v(G)}{2}$, no $p$-proportional graph can exist with $p$ irrational. Conversely, for any rational $p \in (0,1)$ it was shown by K\"arrman \cite{Kar93} and Janson and Spencer \cite{JS92} that there exist infinitely many $p$-proportional graphs. Janson and Spencer \cite{JS92} additionally showed that $n$-vertex $p$-proportional graphs exist for all sufficiently large $n$ satisfying an integrality condition (the same one that we discuss below). Turning to $p$-superproportional graphs, in 1994 K\"arrman \cite{Kar94} published an example of a $64$-vertex $\frac{1}{2}$-superproportional graph found by computer search, and no other examples have appeared in the literature.

Using \cref{thm:main}, we are able to resolve a conjecture of Janson and Kratochv\'il \cite{JK91} that a local central limit theorem exists for proportional graphs. Specifically, given any downwards collection of NIV connected graphs $\calh$, the probability that $G \sim G(n,p)$ is $\suphat{(p, \calh)}$-proportional is
\begin{equation} \label{eq:countpg}
\frac{\prod_{H \in \calh} (\aut H)^{1/2}}{(2\pi)^{\abs{\calh}/2} n^{\sum_{H \in \calh} v(H)/2} (p(1-p))^{\sum_{H \in \calh} e(H)/2}} \cdot (1 + O_{p,\eps}(n^{-1/2+\eps})),
\end{equation}
provided that the all-zero tuple is permissible. To that end, we make the following definition.
\begin{defn} \label{def:ppc}
A positive integer $n \geq 3$ is \vocab{$p$-proportional-compatible} (\vocab{$p$-PC} for short) if the evaluation on $\calr_{\calc_2 \cup \calc_3, n}$ sending $\gamma_H \mapsto 0$ for all $H \in \calc_2 \cup \calc_3$ is integral. A positive integer $n \geq 4$ is \vocab{$p$-superproportional-compatible} (\vocab{$p$-SPC}) if the evaluation on $\calr_{\calc_2 \cup \calc_3 \cup \calc_4, n}$ sending $\gamma_H \mapsto 0$ for all $H \in \calc_2 \cup \calc_3 \cup \calc_4$ is integral.
\end{defn}

With $p$-hyperproportional graphs, the situation is slightly more complicated. Note that since we always have
\[\chi_e^2 + \frac{2p-1}{\sqrt{p(1-p)}} \chi_e - 1 = 0,\]
we conclude that
\begin{equation} \label{eq:2k2}
\gamma_{K_2}^2 = 2\gamma_{2K_2} + 2\gamma_{P_2} - \frac{2p-1}{\sqrt{p(1-p)}} \gamma_{K_2} + \binom{n}{2}.
\end{equation}
Therefore, $G$ is $p$-hyperproportional if and only if $G$ is $\suphat{(p,\calc_3\cup\calc_4\cup\calc_5)}$-proportional and
\begin{equation*} 
\gamma_{K_2}^2 + \frac{2p-1}{\sqrt{p(1-p)}} \gamma_{K_2} - \binom{n}{2} = 0 \iff \gamma_{K_2} = \frac{\frac{1}{2}-p}{\sqrt{p(1-p)}} \pm \sqrt{\binom{n}{2} + \frac{(\frac{1}{2}-p)^2}{p(1-p)}}.
\end{equation*}
Thus there are two types of $p$-hyperproportional graphs: those where the above equality holds with a $+$ sign and those where it holds with a $-$ sign. Call such graphs \vocab{$(p,+)$-hyperproportional} and \vocab{$(p,-)$-hyperproportional}, respectively. Furthermore, define \vocab{$(p,+)$-} and \vocab{$(p,-)$-hyperproportional-compatible} (\vocab{$(p,\pm)$-HPC}) positive integers $n \geq 5$ in a way analogous to \cref{def:ppc}. Now, since in either case we have $\abs{\gamma_{K_2}} \lsim_p n \asymp \sigma_{K_2}$, \cref{thm:main} immediately implies the following.
\begin{cor}
For sufficiently large $(p,+)$-HPC (resp.~$(p,-)$-HPC) $n$ depending on $p$, there exists a $(p,+)$- (resp.~$(p,-)$-) hyperproportional graph with $n$ vertices.
\end{cor}
Moreover, it is straightforward to enumerate such graphs in a manner analogous to \labelcref{eq:countpg}.

We conclude this section with results concerning the properties of $p$-PC, $p$-SPC, and $(p, \pm)$-HPC numbers, with their proofs, consisting of elementary number-theoretic computations, deferred to \cref{sec:app}. We begin with full characterizations of $p$-PC and $p$-SPC numbers, for which there exist infinitely many for every rational $p$.
\begin{prop} \label{prop:pc}
Suppose $p = a/b$ where $a$ and $b$ are relatively prime positive integers. Then a positive integer $n \geq 3$ is $p$-PC if and only if
\begin{cond}
\item \label{cond:pc1} If $2 \mid b$, $\nu_2(n) \geq 3\nu_2(b)$ or $\nu_2(n-1) \geq 3\nu_2(b) + 1$;
\item \label{cond:pc2} If $3 \mid b$, $\max\set{\nu_3(n),\nu_3(n-1)} \geq 3\nu_3(b) + 1$;
\item \label{cond:pc3} For every prime $q > 3$, $\max\set{\nu_q(n),\nu_q(n-1)} \geq 3\nu_q(b)$.
\end{cond}
\end{prop}
\begin{prop} \label{prop:spc}
Suppose $p = a/b$ where $a$ and $b$ are relatively prime positive integers. Then a positive integer $n \geq 4$ is $p$-SPC if and only if
\begin{cond}
\item \label{cond:spc1} If $\nu_2(b) = 1$, $\nu_2(n) \geq 6$ or $\nu_2(n-1) \geq 7$;
\item \label{cond:spc2} If $\nu_2(b) = 2$, $\nu_2(n) \geq 13$ or $\nu_2(n-1) \geq 11$;
\item \label{cond:spc3} If $\nu_2(b) \geq 3$, $\nu_2(n) \geq 6\nu_2(b) + 1$ or $\nu_2(n-1) \geq 6\nu_2(b)$;
\item \label{cond:spc4} If $3 \mid b$, $\nu_3(n) \geq 6\nu_3(b)$ or $\nu_3(n-1) \geq 6\nu_3(b) + 1$;
\item \label{cond:spc5} For every prime $q > 3$, $\max\set{\nu_q(n),\nu_q(n-1)} \geq 6\nu_q(b)$.
\end{cond}
\end{prop}
While it would be theoretically possible to write out a similar characterization of $(p, \pm)$-HPC numbers, there are so many cases that we believe that it is not worth stating. Instead of a full characterization, we observe the following special cases.
\begin{prop} \label{prop:hpcgen}
If $n$ is $(p, \pm)$-HPC, then $p$ must be rational. Moreover, given a rational $p \in (0,1)$, there are infinitely many $(p, +)$-HPC numbers if and only if there are infinitely many $(p, -)$-HPC numbers, which occurs precisely when $p \in \set{\frac{1}{3}, \frac{2}{3}}$ or $p(1-p)$ is not $2c^2$ for some rational $c$.
\end{prop}

As one might guess from the square roots involved in the definition of $(p, \pm)$-HPC numbers, $(p, \pm)$-HPC numbers correspond to solutions to generalized Pell equations subject to certain congruence relations, meaning that they are significantly rarer than $p$-PC or $p$-SPC numbers. As an example, we can characterize $(\frac{1}{2}, \pm)$-HPC numbers as follows.
\begin{prop} \label{prop:hpc12}
A positive integer $n$ is $(\frac{1}{2}, +)$-HPC if and only if it is $(\frac{1}{2}, -)$-HPC. Moreover this happens if and only if
\[n = \paren*{\frac{(1+\sqrt{2})^a + (1+\sqrt{2})^{-a}}{2}}^2\]
for some positive integer $a \geq 2$ with $a \equiv 0, \pm 1, \pm 511 \pmod{1024}$. In particular, the smallest $(\frac{1}{2},\pm)$-HPC number is
\begin{align*}
n &= \begin{array}[t]{@{}r@{}}
393269643023291698757257685885597325993848383834865007942605471587 \\
76646090803634139011571761644665911164995315856589457844040190274 \\
86900324895339884998922974107837617595976120658101454799784430552 \\
76491318398420535797250926457828227049436167142838296079634563380 \\
03268437259421557766468489165196802438714427492861326293694236836 \\
16897572759524717640627107177163613602416684747964902340756531202
\end{array} \\
&\approx 3.93 \times 10^{390}.
\end{align*}
\end{prop}
On the other hand, exactly when $p \in \set{\frac{1}{3},\frac{2}{3}}$, the Pell equation becomes degenerate and the number of solutions increases drastically. In particular, we have the following.
\begin{prop} \label{prop:hpc13}
A positive integer $n \geq 5$ is $(\frac{1}{3}, +)$-HPC if and only if it is $(\frac{2}{3}, -)$-HPC, which occurs if and only if $3^9 \mid n$ or $3^{10} \mid n-2$. A positive integer $n \geq 5$ is $(\frac{1}{3}, -)$-HPC if and only if it is $(\frac{2}{3}, +)$-HPC, which occurs if and only if $3^{10} \mid n-1$ or $3^{10} \mid n-2$.
\end{prop}
Although we do not have a rigorous proof, it appears likely that the only HPC integers less than, say, $10^{100}$ are those mentioned above. While working with dense $3^9 = 19683$-vertex graphs is possible with computers, actually finding a $19683$-vertex $\frac{1}{3}$-hyperproportional graph appears to be out of reach barring some algorithmic improvement.

\appendix
\section{On \texorpdfstring{$p$}{p}-PC, \texorpdfstring{$p$}{p}-HPC, and \texorpdfstring{$(p,\pm)$}{(p, \textpm)}-HPC Numbers} \label{sec:app}
We begin by stating two relations in $\calr_n$. First of all, we recall \labelcref{eq:xgamma}:
\begin{equation} \label{eq:xtogammans}
X_{H} = \sum_{H' \subseteq H}  p^{e(H) - e(H')} (p(1-p))^{e(H')/2}   \frac{\aut H'}{\aut H} \gamma_{H'}.
\end{equation}
Moreover, if $H$ is NIV and $m \geq 0$, we have
\begin{equation} \label{eq:addiv}
\gamma_{H \sqcup m\kone} = \binom{n-v(H)}{m} \gamma_H.
\end{equation}
\subsection{Proof of \texorpdfstring{\cref{prop:pc}}{Proposition \ref{prop:pc}}}
Every graph with at most $3$ vertices either has no edges or is the disjoint union of an element of $\calc_2 \cup \calc_3$ and some isolated vertices. Therefore, if $\Phi$ sends $\gamma_H \mapsto 0$ for all $H \in \calc_2 \cup \calc_3$, we find that for all $H \in \calc_2 \cup \calc_3$, all but one term in \labelcref{eq:xtogammans} is zero. Thus
\[\Phi(X_H) = p^{e(H)}\frac{v(H)!}{\aut H} \binom{n}{v(H)}.\]
Since $\frac{v(H)!}{\aut H} \binom{n}{v(H)}$ is an integer, we conclude that $n$ is $p$-PC if and only if $b \mid \binom{n}{2}$, $b^2 \mid 3\binom{n}{3}$, and $b^3 \mid \binom{n}{3}$. Note that the third condition obviously implies the second. Thus $n$ is $p$-PC if and only if
\[\nu_q(n) + \nu_q(n-1) \geq \nu_q(b) + \nu_q(2) \quad\text{and}\quad \nu_q(n) + \nu_q(n-1) + \nu_q(n-2) \geq 3\nu_q(b) + \nu_q(6)\]
for every prime $q$ that divides $b$.

If $q = 2$, the first condition implies that $n \equiv 0,1 \pmod{4}$. If $n \equiv 0 \pmod{4}$, the conditions are satisfied if and only if we additionally have $\nu_2(n) \geq 3\nu_2(b)$. If $n \equiv 1 \pmod{4}$, the conditions are equivalent to $\nu_2(n-1) \geq 3\nu_2(b) + 1$. This yields \cref{cond:pc1}.

If $q \geq 3$, and $n$ is not $0$ or $1$ mod $q$, the first condition yields a contradiction. Otherwise, if $n \equiv r \in \set{0,1}$ modulo $q$, then the conditions are equivalent to $\nu_q(n-r) \geq 3\nu_q(b) + \nu_q(3)$, yielding \cref{cond:pc2,cond:pc3}.

\subsection{Proof of \texorpdfstring{\cref{prop:spc}}{Proposition \ref{prop:spc}}}
Let $\Phi$ be the evaluation of $\calr_{\calc_2 \cup \calc_3 \cup \calc_4,n}$ sending $\gamma_H \mapsto 0$ for all $H \in \calc_2 \cup \calc_3 \cup \calc_4$. By \labelcref{eq:addiv}, we find that for all $H$ with at most $4$ vertices, we have $\Phi(\gamma_H) = 0$, unless $H$ is edgeless or $H = 2K_2$. In the latter case, by \labelcref{eq:2k2} we find that
\[\Phi(\gamma_{2K_2}) = -\frac{1}{2} \binom{n}{2}.\]
Therefore, by \labelcref{eq:xtogammans}, we find that for all $H \in \calc_2 \cup \calc_3 \cup \calc_4$,
\[\Phi(X_H) = p^{e(H)} \frac{v(H)!}{\aut H} \binom{n}{v(H)} - p^{e(H)-1}(1-p) \frac{4X_{2K_2}(H)}{\aut H} \binom{n}{2}.\]
Therefore, $n$ is $p$-SPC if and only if the following nine quantities are integers:
\begin{gather*}
\begin{aligned}
\Phi(X_{K_2}) &= p\binom{n}{2} &
\Phi(X_{P_2}) &= 3p^2 \binom{n}{3} &
\Phi(X_{K_3}) &= p^3 \binom{n}{3} &
\Phi(X_{K_{1,3}}) &= 4p^3\binom{n}{4}
\end{aligned} \\
\begin{aligned}
\Phi(X_{P_3}) &= 12p^3 \binom{n}{4} - 2p^2(1-p) \binom{n}{2} &
\Phi(X_{\overline{P_2 \sqcup \kone}}) &= 12p^4\binom{n}{4} - 2p^3(1-p) \binom{n}{2} \\
\Phi(X_{C_4}) &= 3p^4\binom{n}{4} - p^3(1-p) \binom{n}{2} &
\Phi(X_{\overline{K_2\sqcup 2\kone}}) &= 6p^5 \binom{n}{4} - 2p^4(1-p)\binom{n}{2} \\
\Phi(X_{K_4}) &= p^6 \binom{n}{4} - \frac{1}{2}p^5(1-p) \binom{n}{2} &&
\end{aligned}
\end{gather*}

Suppose $n$ is $p$-SPC and consider the quantity $N = 2p^5\binom{n}{4}-p^5(1-p)\binom{n}{2}$. If $n$ is $p$-SPC, $N$ has to be an integer as well, since the only way for it not to be an integer is that if it has factors of $b$ in its denominator, and if that is the case then $pN = 2\Phi(X_{K_4})$ cannot be an integer. Therefore,
\[\Phi(X_{\overline{K_2\sqcup 2\kone}}) - 3N = p^4(1-p) \binom{n}{2}\]
has to be an integer as well, implying that $b^5 \mid \binom{n}{2}$.

Conversely, if $b^5 \mid \binom{n}{2}$, it follows that
\[b^5 \mid 3\binom{n}{3} = (n-2)\binom{n}{2}\quad\text{and}\quad b^5 \mid 6\binom{n}{4} = \binom{n}{2} \binom{n-2}{2}.\] It follows that that $b^4 \mid \binom{n}{3}, \binom{n}{4}$, implying for all $H \in \calc_2 \cup \calc_3 \cup \calc_4$ with $e(H) \leq 4$, $\Phi(X_H)$ is an integer. Moreover,
\[\Phi(X_{\overline{K_2\sqcup 2\kone}}) = 6p^5 \binom{n}{4} - 2p^4(1-p)\binom{n}{2} = p^5 \binom{n}{2} \binom{n-2}{2} - 2p^4(1-p)\binom{n}{2}\]
is an integer as well. As a result, we find that $n$ is $p$-SPC if and only if $b^5 \mid \binom{n}{2}$ and
\[\Phi(X_{K_4}) \in \setz \iff 2b^6 \mid a^5 \paren*{2a\binom{n}{4} - (b-a) \binom{n}{2}}.\]
We now take $q$-adic evaluations, where $q$ ranges across the prime divisors of $2b$.

First, if $q \geq 3$, then $\nu_q(b^5) \leq \nu_q(\binom{n}{2})$ is equivalent to $\max \set{\nu_q(n), \nu_q(n-1)} \geq 5\nu_q(b)$. Moreover, since $b^6 \mid b\binom{n}{2}$, the $q$-adic portion of the second condition is equivalent to
\[6\nu_q(b) \leq \nu_q \paren*{a^5\paren*{2a\binom{n}{4}+a\binom{n}{2}}} = \nu_q \paren*{2\binom{n}{4} + \binom{n}{2}}.\]
Now, we may compute
\[2\binom{n}{4} + \binom{n}{2} = \binom{n}{2} \paren*{\frac{(n-2)(n-3) + 6}{6}}.\]
If $n \equiv 0,1 \pmod{q^{5\nu_q(b)}}$, we also have $(n-2)(n-3) + 6 \equiv 12, 8 \pmod{q^{5\nu_q(b)}}$ (respectively). Thus $\nu_q((n-2)(n-3) + 6)$ is $0$ unless $q = 3$ and $n \equiv 0 \pmod{q^{5\nu_q(b)}}$, in which case it is $1$, meaning that
\[\nu_q \paren*{2\binom{n}{4} + \binom{n}{2}} = \nu_q\paren*{\binom{n}{2}} + \begin{cases*}
-1 & if $q = 3$ and $n \equiv 1 \pmod{q^{5\nu_q(b)}}$ \\
0 & else
\end{cases*}\]
Thus our final condition is $\max\set{\nu_q(n), \nu_q(n-1)} \geq 6\nu_q(b)$ if $q \geq 5$, and $\nu_q(n) \geq 6\nu_q(b)$ or $\nu_q(n-1) \geq 6\nu_q(b) + 1$ if $q = 3$, matching \cref{cond:spc4,cond:spc5}.

We now turn to the case $q = 2$. First of all, if $b$ is odd, then the conditions are equivalent to
\[a^5\paren*{2a\binom{n}{4}-(b-a)\binom{n}{2}}\]
being even. The first term is obviously even, while the second is as well since $a(b-a)$ must be even. Henceforth assume that $b$ is even. Since $a$ is odd, we find that assuming $5\nu_2(b) \leq \nu_2(\binom{n}{2})$,
\[\nu_2(2b^6) \leq \nu_2\paren*{b(a/3+1)\binom{n}{2}},\]
meaning that we can replace the second condition with
\[6\nu_2(b) + 1 \leq \nu_2\paren*{a^5 \paren*{2a\binom{n}{4} - (-ab/3-a) \binom{n}{2}}} = \nu_2\paren*{2\binom{n}{4} + (b/3+1) \binom{n}{2}}.\]
Now, since
\[2\binom{n}{4} + (b/3+1) \binom{n}{2} = \binom{n}{2} \paren*{\frac{(n-2)(n-3)+6+2b}{6}},\]
we have rewritten our conditions as
\[5\nu_2(b) \leq \nu_2\paren*{\binom{n}{2}}\quad\text{ and }\quad 6\nu_2(b) + 2 \leq \nu_2\paren*{\binom{n}{2}} + \nu_2((n-2)(n-3)+6+2b).\]

Now assume that $n$ is even. Then $5\nu_2(b) + 1 \leq \nu_2(n)$, meaning that modulo $2^{5\nu_2(b)+1}$, we have $(n-2)(n-3)+6+2b \equiv 2b+12$, so $(n-2)(n-3)+6+2b$ has $2$-adic evaluation at least $3$ if $\nu_2(b) = 1$ and exactly $2$ otherwise. Thus the conditions are satisfied exactly when $\nu_2(n) \geq 6$ if $\nu_2(b) = 1$ and $\nu_2(n) \geq 6\nu_2(b)+1$ otherwise, matching \cref{cond:spc1,cond:spc2,cond:spc3}.

If $n$ is odd, then $5\nu_2(b) + 1 \leq \nu_2(n-1)$, so modulo $2^{5\nu_2(b)+1}$, we have $(n-2)(n-3)+6+2b \equiv 2b+8$, meaning that the $2$-adic evaluation of $(n-2)(n-3)+6+2b$ is $2$ if $\nu_2(b) = 1$, at least $4$ if $\nu_2(b) = 2$, and $3$ otherwise. Therefore the conditions are satisfied if and only if $\nu_2(n-1) \geq 7$ for $\nu_2(b) = 1$, $\nu_2(n-1) \geq 11$ for $\nu_2(b) = 2$, and $\nu_2(n-1) \geq 6\nu_2(b)$ otherwise. This also matches \cref{cond:spc1,cond:spc2,cond:spc3}.

\subsection{Generalities surrounding \texorpdfstring{\boldmath $(p,\pm)$}{(p, \textpm)}-HPC numbers}
In this section we develop a number of general results that will be useful in the proofs of \cref{prop:hpcgen,prop:hpc12,prop:hpc13}, including proving part of \cref{prop:hpcgen}.

Consider the evaluation $\Phi_\pm$ sending $\gamma_H \mapsto 0$ for all $H \in \calc_3 \cup \calc_4 \cup \calc_5$ and sending
\[\gamma_{K_2} \mapsto \frac{\frac{1}{2}-p}{\sqrt{p(1-p)}} \pm \sqrt{\binom{n}{2} + \frac{(\frac{1}{2}-p)^2}{p(1-p)}}.\]
First we compute $\Phi_\pm(\gamma_H)$ for all NIV $H$. The NIV graphs with at most $5$ vertices consist of the empty graph, $K_2$, members of $\calc_3 \cup \calc_4 \cup \calc_5$, $2K_2$, $K_2 \sqcup P_2$, and $K_2 \sqcup K_3$. By construction, $\Phi_\pm(\gamma_{2K_2}) = 0$. Also, in a similar manner to \labelcref{eq:2k2} we obtain that in $\calr_n$,
\begin{align*}
\gamma_{K_2} \gamma_{P_2} &= \gamma_{K_2\sqcup P_2} + 2\gamma_{P_3} + 3\gamma_{K_{1,3}} - 2\frac{2p-1}{\sqrt{p(1-p)}}\gamma_{P_2} + 2(n-2)\gamma_{K_2} \\
\gamma_{K_2} \gamma_{K_3} &= \gamma_{K_2\sqcup K_3} + \gamma_{\overline{P_2\sqcup \kone}} - 3\frac{2p-1}{\sqrt{p(1-p)}}\gamma_{K_3} + \gamma_{P_2}.
\end{align*}
It follows that
\[\Phi_\pm(\gamma_{K_2 \sqcup P_2}) = -2(n-2)\Phi_\pm(\gamma_{K_2})\quad\text{and}\quad\Phi_\pm(\gamma_{K_2\sqcup K_3}) = 0.\]
Therefore, from \labelcref{eq:xtogammans} and \labelcref{eq:addiv}, we obtain that for all $H \in \calc_2 \cup \calc_3 \cup \calc_4 \cup \calc_5$, we have
\begin{multline} \label{eq:phiexpansion}
\Phi_\pm(X_H) = p^{e(H)} \frac{v(H)!}{\aut H} \binom{n}{v(H)} + p^{e(H)-1} \sqrt{p(1-p)} \frac{2(v(H)-2)! e(H)}{\aut H} \binom{n-2}{v(H)-2}\Phi_\pm(\gamma_{K_2}) \\ - p^{e(H)-3} (p(1-p))^{3/2} \frac{4 X_{K_2\sqcup P_2}(H)}{\aut H} 2(n-2) \Phi_\pm(\gamma_{K_2}),
\end{multline}
where the last term only occurs for $5$-vertex graphs.

We first need a lemma about these coefficients.
\begin{lem} \label{lem:div}
If $H'$ and $H$ are graphs with $v(H) \geq v(H')$, then
\[\frac{\aut H' \cdot X_{H'}(H) \cdot (v(H)-v(H'))!}{\aut H} \in \setz.\]
In particular, for $H \in \calc_2\cup\calc_3\cup\calc_4\cup\calc_5$,
\[\frac{v(H)!}{\aut H}, \quad \frac{2(v(H)-2)! e(H)}{\aut H},\quad \frac{4X_{K_2\sqcup P_2}(H)}{\aut H}\in \setz\]
(the above statement applied for $H' = \emptyset, K_2, K_2 \sqcup P_2$).
\end{lem}
\begin{proof}
Consider the action of $\Aut H$ on injective graph homomorphisms $H' \to H$. The stabilizer of each homomorphism must permute the vertices not in the image of the homomorphism, meaning that it must have size dividing $(v(H) - v(H'))!$. This implies that the size of each orbit must have size a multiple of $\aut H/(v(H) - v(H'))!$. But there are exactly $\aut H' \cdot X_{H'}(H)$ such homomorphisms, so the result follows.
\end{proof}

We now start proving \cref{prop:hpcgen}.
\begin{claim} \label{claim:irrat}
If there exists an integer $n \geq 5$ that is $(p,+)$- or $(p,-)$-HPC, then $p$ is rational.
\end{claim}
\begin{proof}
First, consider some connected $H$ with at most $4$ vertices. Then, we find that
\begin{align*}
\frac{\aut H}{v(H)! \binom{n}{v(H)}} \Phi_\pm(X_H) &= p^{e(H)} + p^{e(H)-1}\sqrt{p(1-p)} \frac{2e(H)}{n(n-1)} \Phi_\pm(\gamma_{K_2})\\
&= p^{e(H)} \paren*{1 + e(H) \sqrt{\frac{1-p}{p}} \frac{2\Phi_\pm(\gamma_{K_2})}{n(n-1)}}
\end{align*}
is a rational number; since it depends only on $e(H)$ call it $\xi_{e(H)}$. Now we compute that
\[\xi_1^2 - \xi_2 = \frac{4p(1-p)}{n^2(n-1)^2}\Phi_\pm(\gamma_{K_2})^2 \quad\text{ and }\quad \xi_1\xi_2 - \xi_3 = \frac{8p^2(1-p)}{n^2(n-1)^2} \Phi_\pm(\gamma_{K_2})^2\]
are both rational, so their quotient, which is equal to $2p$, must be rational as well.
\end{proof}
\begin{rmk}
This argument in fact proves that if there exists a $\suphat{(p, \calc_3)}$-proportional graph, then $p$ is rational.
\end{rmk}

For the remainder of the proof we will consider rational $p$; suppose that $p = a/b$ with $\gcd(a,b) = 1$. Define
\[h_\pm = \sqrt{a(b-a)}\Phi_\pm(\gamma_{K_2}) = \frac{b - 2a \pm \sqrt{D}}{2},\]
where we additionally define
\[D = 2a(b-a)n(n-1) + (b-2a)^2 = \frac{2a(b-a)(2n-1)^2+2(3a-b)(3a-2b)}{4}.\]
Now \labelcref{eq:phiexpansion} becomes
\begin{multline} \label{eq:phiab}
\Phi_\pm(X_H) = \frac{a^{e(H)-2}}{b^{e(H)}} \paren[\bigg]{\frac{v(H)!}{\aut H} a^2 \binom{n}{v(H)} +  \frac{2(v(H)-2)! e(H)}{\aut H} a \binom{n-2}{v(H)-2} h_\pm \\ -  2\frac{4X_{K_2\sqcup P_2}(H)}{\aut H} (b-a)(n-2)h_\pm}.
\end{multline}
\begin{lem} \label{lem:hpcabc}
A necessary condition for $n$ to be $(p, \pm)$-HPC is that $\frac{1}{b}(h_\pm+a\binom{n}{2})$ is an integer and $b^8 \mid 2h_\pm(n-2)$. A sufficient condition for $n$ to be $(p, \pm)$-HPC is that $\frac{1}{b}(h_\pm+a\binom{n}{2})$ is an integer and $b^{12} \mid n-2$.
\end{lem}
\begin{proof}
First assume that $n$ is $(p, \pm)$-HPC. Since
\[\Phi_\pm(X_{K_2}) = \frac{a}{b} \binom{n}{2} + \frac{\sqrt{a(b-a)}}{b} \Phi_\pm(\gamma_{K_2}) = \frac{a\binom{n}{2} + h_\pm}{b},\]
the first condition is evident. Now, we consider the graphs $K_5$, $\overline{K_2 \sqcup 3\kone}$, and $-(P_2 + 2\kone)$. After computing
\begin{align*}
\aut K_5 &= 120 & \aut (\overline{K_2 \sqcup 3\kone}) &= 12 & \aut(\overline{P_2 \sqcup 2\kone}) &= 4 \\
e(K_5) &= 10 & e(\overline{K_2 \sqcup 3\kone}) &= 9 & e(\overline{P_2 \sqcup 2\kone}) &= 8 \\
X_{K_2\sqcup P_2}(K_5) &= 30 &  X_{K_2\sqcup P_2}(\overline{K_2 \sqcup 3\kone}) &= 21 & X_{K_2\sqcup P_2}(\overline{P_2 \sqcup 2\kone}) &= 13
\end{align*}
by \labelcref{eq:phiexpansion} we have
\begin{align}
\Phi_\pm(X_{K_5}) &= \frac{a^8}{b^{10}} \paren*{a^2\binom{n}{5} + a\binom{n-2}{3} h_\pm - 2(b-a)(n-2)h_\pm} \label{eq:xk5} \\
\Phi_\pm(X_{\overline{K_2 \sqcup 3\kone}}) &= \frac{a^7}{b^{9}} \paren*{10a^3\binom{n}{5} + 9a\binom{n-2}{3} h_\pm - 14(b-a)(n-2)h_\pm} \label{eq:xk5me} \\
\Phi_\pm(X_{\overline{P_2 \sqcup 2\kone}}) &= \frac{a^6}{b^{8}} \paren*{30a^3\binom{n}{5} + 24a\binom{n-2}{3} h_\pm - 26(b-a)(n-2)h_\pm}. \nonumber
\end{align}
Now we observe that
\[-30b^2\Phi_\pm(X_{K_5}) + 6ab\Phi_\pm(X_{\overline{K_2 \sqcup 3\kone}}) - a^2\Phi_\pm(X_{\overline{P_2 \sqcup 2\kone}}) = \frac{a^8}{b^8} (2(b-a)(n-2)h_\pm),\]
so $b^8 \mid 2(n-2)h_\pm$.

Conversely, assume that $\frac{1}{b}(h_\pm+a\binom{n}{2})$ is an integer and $b^{12} \mid n-2$. Clearly $\Phi_\pm(X_{K_2})$ and $h_\pm$ are integers. Moreover, for $H \in \calc_3 \cup \calc_4 \cup \calc_5$, it is straightforward to show that $b^{10}$ divides $\binom{n}{v(H)}$, $\binom{n-2}{v(H)-2}$, and $n-2$. We are now done by combining \labelcref{eq:phiab} and \cref{lem:div}.
\end{proof}

\subsection{Proof of remainder of \texorpdfstring{\cref{prop:hpcgen}}{Proposition \ref{prop:hpcgen}}}
In light of \cref{claim:irrat}, it suffices to consider rational $p$.

If $p \notin \set{\frac{1}{3},\frac{2}{3}}$ and $2a(b-a)$ is a square, then there are only finitely many $n$ for which $4D$ is a perfect square, since $2(3a-b)(3a-2b) \neq 0$, the number $2a(b-a)(2n-1)^2$ is a square, and there are only finitely many pairs of perfect squares with a given nonzero difference. Thus there are only finitely many $n$ such that $h_\pm$ is an integer, so by \cref{lem:hpcabc} there are only finitely many $(p, \pm)$-HPC $n$.

Now suppose $p$ is rational but not of the above form. We first need a lemma.
\begin{lem}
For every positive integer $m$, there exist infinitely many positive integers $n \equiv 2 \pmod{m}$ such that $D$ is a square.
\end{lem}
\begin{proof}
If $p \in \set{\frac{1}{3}, \frac{2}{3}}$, we may compute that $D = (2n-1)^2$, so the conclusion is obvious. For other $p$, we will use the theory of Pell equations. In particular, we know that since $2a(b-a)$ is not a perfect square, there exist positive integers $r$ and $s$ such that $r^2 - 2a(b-a)s^2 = 1$. We claim that we may additionally choose $r$ and $s$ to be equivalent to $1$ and $0$ mod $2m$. To see this, observe that $r + s\sqrt{2a(b-a)}$ is a unit inside the ring $R = (\setz/2m\setz)[\sqrt{2a(b-a)}]$. Since $R$ is finite, some power of $r + s\sqrt{2a(b-a)}$ must be equal to $1$ in $R$, yielding the desired solution. Now, since
\[(2b)^2 - 2a(b-a) \cdot 3^2 = 2(3a-b)(3a-2b)\]
(a reflection of the fact that $K_2$ would be hyperproportional were it not too small), by multiplying $2b + 3\sqrt{2a(b-a)}$ by arbitrarily large powers of $r + s\sqrt{2a(b-a)}$ we find arbitrarily large positive integers $t$ and $u$, congruent to $2b$ and $3$ mod $2m$, such that
\[t^2 - 2a(b-a) \cdot u^2 = 2(3a-b)(3a-2b).\]
We now claim that $n = \frac{u+1}{2}$ works. Indeed, by construction $u \equiv 2 \pmod{m}$ and $D = t^2/4$. Since $t$ is even, we are done.
\end{proof}

We now claim that if $n \equiv 2 \pmod{2b^{12}}$ and $D$ is a square, $n$ is both $(p,+)$- and $(p,-)$-HPC. First of all, we note that $D \equiv 4a(b-a) + (b-2a)^2 \equiv b^2 \pmod{2b^2}$, so $\sqrt{D}$ is $b$ times an odd integer. Therefore $h_\pm$ is an integer that is $-a$ mod $b$. Moreover, $n(n-1) \equiv 2 \pmod{2b}$, so $\binom{n}{2} \equiv 1 \pmod{b}$. Therfore $\frac{1}{b}(h_\pm + \binom{n}{2})$ is an integer, meaning that we are done by \cref{lem:hpcabc}.

\subsection{Proof of \texorpdfstring{\cref{prop:hpc12}}{Proposition \ref{prop:hpc12}}}
The proof is in two parts: we first reduce to a divisibility condition on $n$, and then solve the Pell equation given these divisiblity conditions. We will also note that $h_\pm = \sqrt{\binom{n}{2}}$.
\begin{claim}
An integer $n \geq 5$ is $(\frac{1}{2},+)$-HPC if and only if it is $(\frac{1}{2},-)$-HPC, which occurs exactly when
\begin{itemize}
\item $h_\pm$ is an integer, and 
\item $2^{13} \mid n$, $2^{21} \mid n-1$, or $2^{11} \mid n-2$.
\end{itemize}
\end{claim}
\begin{proof}
Throughout this proof, we will frequently use \cref{lem:div} implicitly. Also note that if $h_\pm$ is an integer, then $\nu_2(h_\pm) = \frac{1}{2} (\max\set{\nu_2(n),\nu_2(n-1)}-1)$.

Assume $h_\pm$ is an integer. If $2^{13} \mid n$, then we always have $2^{10} \mid \binom{n}{v(H)}$, so the first term in \labelcref{eq:phiab} is an integer. Also $2^6 \mid h_\pm$, so the second and third terms are integers for all graphs with at most $4$ vertices. For $5$ vertex graphs, we note that $\binom{n-2}{3} \equiv -4 \pmod{2^{12}}$, so $\binom{n-2}{3} h_\pm$ and $2(n-2) h_\pm$ are both divisible by $2^8$. It remains to check that for connected $5$-vertex graphs $H$ with at least $9$ edges,
\[2^{e(H) - 8} \mid -\frac{12 e(H)}{\aut H} + \frac{4 X_{K_2\sqcup P_2}(H)}{\aut H}.\]
If $H$ is $K_5$ minus an edge, then we need to show $2 \mid -9+7$, which is true. If $H = K_5$, then we need to show $4 \mid -1 + 1$, which is also true.

If $2^{21} \mid n-1$, then $2^{10} \mid h_\pm$ and $2^{10} \mid \binom{n}{v(H)}$, so $n$ is indeed $(\frac{1}{2},\pm)$-HPC.

If $2^{11} \mid n-2$, then we first note that since $h_\pm$ is odd, $\Phi_\pm(X_{K_2}) = \frac{1}{2}(\binom{n}{2} + h_\pm)$ is an integer. Now, for $v(H) \geq 3$, we have $2^{10} \mid \binom{n-2}{v(H) - 2}$, so we may ignore the second and third terms. Moreover, it is straightforward to compute that $2^{11} \mid \binom{n}{3}$, $2^{9} \mid \binom{n}{4}$, and $2^{10} \mid \binom{n}{5}$, so the first term is integral as well. This finishes the proof of one direction.

In the other direction, we note that by \cref{lem:hpcabc} we get that $h_\pm$ is an integer and $2^7 \mid h_\pm (n-2)$. We now split into cases depending on $n$ mod $4$.

If $n \equiv 0 \pmod 4$, then we get that $2^6 \mid h_\pm$, so $2^{13} \mid n$, as desired.

If $n \equiv 1 \pmod 4$, then $2^7 \mid h_\pm$ and thus $2^{15} \mid n-1$. This implies that $2^{10} \mid \binom{n}{5}$, so by considering $\Phi_\pm(X_{K_5})$ we get that
\[2^{10} \mid h_\pm \paren*{\binom{n-2}{3} - 2(n-2)}.\]
But $\binom{n-2}{3} - 2(n-2)$ is odd, so $2^{10} \mid h_\pm$ and $2^{21} \mid n-1$.

If $n \equiv 2 \pmod 4$, then $h_\pm$ is odd, immediately implying that $2^7 \mid n-2$. This is enough to imply that
\begin{equation*}
\nu_2 \paren*{\binom{n}{5}} = \nu_2(n-2) - 1, \qquad
\nu_2 \paren*{\binom{n-2}{3}h_\pm} = \nu_2(n-2),  \qquad
\nu_2 \paren*{2(n-2)h_\pm} = \nu_2(n-2) + 1.
\end{equation*}
Therefore $\nu_2(\Phi_\pm(X_{K_5})) = \nu_2(n-2) - 11$, implying that $2^{11} \mid n-2$.

Finally, if $n \equiv 3 \pmod 4$, then $(n-2)h_\pm$ is odd, contradiction. This concludes the proof.
\end{proof}

Now, since we can rewrite $h_\pm = \pm \sqrt{\binom{n}{2}}$ as $(2n-1)^2 - 2(2h_\pm)^2 = 1$, the values of $n \geq 1$ with $h_\pm$ integer correspond to solutions to the Pell equation $r^2 - 2s^2 = 1$ with $r$ odd and positive and $s$ even and nonnegative (which is all nonnegative solutions by mod $4$ reasons). The fundamental solution is $(r,s) = (3,2)$, so in general the nonnegative solutions are given by $(r_a, s_a)$ where
\[r_a + s_a\sqrt{2} = (3 + 2\sqrt{2})^a = (1 + \sqrt{2})^{2a} \iff r_a = \frac{(1 + \sqrt{2})^{2a} + (1 + \sqrt{2})^{-2a}}{2}\]
for $a \geq 0$. This in turn means that
\[n = \frac{r_a+1}{2} = \paren*{\frac{(1 + \sqrt{2})^{a} + (1 + \sqrt{2})^{-a}}{2}}^2.\]
It remains to show that for $n \geq 5$ (equivalent to $a \geq 2$) of this form, we have $2^{13} \mid n$, $2^{21} \mid n-1$, or $2^{11} \mid n-2$ if and only if $a \equiv 0, \pm 1, \pm 511 \pmod{1024}$. In other words, we need to show the following.
\begin{claim}
Let $a \geq 2$ and let $r_a + s_a\sqrt{2} = (3 + 2\sqrt{2})^a$. The following are equivalent:
\begin{cond}
\item \label{cond:p1} $2^{14} \mid r_a + 1$, $2^{22} \mid r_a - 1$, or $2^{12} \mid r_a - 3$,
\item \label{cond:p2} $a \equiv 0, \pm 1, \pm 511 \pmod{1024}$.
\end{cond}
\end{claim}
\begin{proof}
We work inside the field $\setq_2(\sqrt{2})$, which has a valuation $\Abs{-}_2$ and ring of integers $\setz_2[\sqrt{2}]$. Let $\rho = 3 + 2\sqrt{2}$.

\makeatletter\@hyper@itemtrue\def\@currentlabelname{}\makeatother\refstepcounter{condi}\label{cond:p3}
We add a condition (3): that $\rho^a \equiv \rho \pmod{2^{10}}$, $\rho^a \equiv \rho\inv \pmod{2^{10}}$, or $\rho^a \equiv 1 \pmod{2^{11}}$.

We first show that \labelcref{cond:p1} and \labelcref{cond:p3} are equivalent. We first claim that we cannot have $r_a \equiv 5, 7 \pmod{8}$. Indeed, since $\rho^2 = 17 + 12\sqrt{2}$, it follows that the powers of $\rho$ alternate $3 + 2\sqrt{2}$ and $1$ modulo $4\sqrt{2}$. So $2^{14} \mid r_a + 1$ is in fact impossible. Now, if $2^{22} \mid r_a - 1$, then $r_a^2 \equiv 1 \pmod{2^{23}}$, implying that $s_a^2 \equiv 0 \pmod{2^{22}}$ and thus $2^{11} \mid s_a$. Therefore $\rho^a \equiv 1 \pmod{2^{11}}$. Moreover, if $r_a \equiv 3 \pmod{2^{12}}$, then $r_a^2 \equiv 9 \pmod{2^{13}}$, implying that $s_a^2 \equiv 4 \pmod{2^{12}}$, which then implies that $s_a \equiv \pm 2 \pmod{2^{10}}$. Thus $\rho^a \equiv \rho^{\pm 1} \pmod{2^{10}}$. This proves that \labelcref{cond:p1} implies \labelcref{cond:p3}.

The opposite is similar. If $\rho^a \equiv 1 \pmod{2^{11}}$, then $2^{11} \mid s_a$, proving that $2^{23} \mid r_a^2 - 1$ and thus $2^{22} \mid r_a \pm 1$. But it is impossible to have $r_a \equiv 7 \pmod{8}$, so $2^{22} \mid r_a - 1$. If $\rho^a \equiv \rho^{\pm  1}\pmod{2^{10}}$, then $s_a^2 \equiv 4 \pmod{2^{12}}$, so $r_a^2 \equiv 9 \pmod{2^{13}}$ and thus $r_a \equiv \pm 3 \pmod{2^{12}}$. But since $r_a \not\equiv 5 \pmod{8}$, we must have $r_a \equiv 3 \pmod{2^{12}}$.

To show that \labelcref{cond:p2} and \labelcref{cond:p3} are equivalent, we note that $\rho^a \equiv \rho^b \pmod{2^{e}}$ if and only if $\rho^{a-b} \equiv 1 \pmod{2^e}$. So it suffices to show that $\rho$ has order $2^9$ mod $2^{10}$ and order $2^{10}$ mod $2^{11}$. To see this, observe that if $\rho$ has order $2$ mod $4$, so its order must be even. Moreover, since $\Abs{\rho^2-1}_2 = 2^{-5/2} < 2^{-1/(2-1)}$, it follows that for even $a$
\[\Abs{\rho^{a} - 1}_2 = \Abs{\log(\rho^{a})}_2 = \Abs{a/2}_2 \Abs{\log(\rho^2)}_2 = \Abs{a/2}_2 \Abs{\rho^2 - 1}_2 = 2^{-3/2}\Abs{a}_2.\]
The result follows.
\end{proof}
We remark that the above claim can also be proven by computer, since $r_a$ is periodic with respect to any modulus.

\subsection{Proof of \texorpdfstring{\cref{prop:hpc13}}{Proposition \ref{prop:hpc13}}}
We first observe three facts. First, by \labelcref{eq:xk5} and \labelcref{eq:xk5me}, we have
\[10b\Phi_\pm(X_{K_5}) - a\Phi_\pm(X_{\overline{K_2\sqcup 3\kone}}) = \frac{a^8}{b^9} \paren*{a\binom{n-2}{3}h_\pm - 6(b-a)(n-2)h_\pm},\]
so by \cref{lem:hpcabc}, a necessary condition for $n$ to be $(p, \pm)$-HPC is $3^9 \mid h_\pm \binom{n-2}{3}$.

Moreover, one can compute the following:
\begin{align}
n \equiv 0 \pmod{3^9} &\implies \binom{n}{5} \equiv \frac{n}{5},\quad \binom{n-2}{3}n \equiv -4n \pmod{3^{10}} \label{eq:binom0} \\
n \equiv 1 \pmod{3^9} &\implies \binom{n}{5} \equiv -\frac{n-1}{20} ,\quad \binom{n-2}{3}(n-1) \equiv -(n-1) \pmod{3^{10}} \label{eq:binom1} \\
n \equiv 2 \pmod{3^{10}} &\implies \binom{n}{5} \equiv \frac{n-2}{30},\quad \binom{n-2}{3} \equiv \frac{n-2}{3} \pmod{3^{10}} \label{eq:binom2}
\end{align}

Finally, note that when $b = 3$, $D = (2n-1)^2$, meaning that for $n \geq 1$, we have $h_+ = n$ and $h_- = 1-n$ when $p = \frac{1}{3}$ and $h_+ = n-1$ and $h_- = -n$ if $p = \frac{2}{3}$. 

We now split into cases.
\subsubsection{Case 1: $n \equiv 0 \pmod 3$}
We aim to show that $n$ is $(\frac{1}{3}, +)$- and $(\frac{2}{3},-)$-HPC if and only if $3^{9} \mid n$, and that $n$ is never $(\frac{1}{3}, -)$- and $(\frac{2}{3},+)$-HPC. The latter fact follows immediately from the fact that $3^8 \mid h_\pm(n-2)$. Now assume that we are working with the other two choices of $p$ and sign, so that $h_\pm = \pm n$.

For the only if direction, note that the condition $3^9 \mid h_\pm \binom{n-2}{3}$ implies that $3^9 \mid n$. For the if direction, note that we automatically have that $3^8$ divides $\binom{n}{v(H)}$, $3^9$ divides $\binom{n}{5}$, and $3^9$ divides $h_\pm$, so it suffices to check $H = K_5$. Since $3^9$ divides $n$ and $h_\pm$, this reduces to checking that modulo $3^{10}$
\[0 \equiv a^2 \binom{n}{5} \pm a\binom{n-2}{3} n \mp 4an \equiv \binom{n}{5} + \binom{n-2}{3}n - 4n.\]
This follows from \labelcref{eq:binom0}.

\subsubsection{Case 2: $n \equiv 1 \pmod 3$}
As in the previous case, from the fact that $3^8 \mid h_\pm(n-2)$ we find that $n$ is never $(\frac{1}{3},+)$- or $(\frac{2}{3},-)$-HPC. Now assume that we are in the other two cases, so that $h_\pm = \pm(n-1)$.

For the only if direction, note that the condition $3^9 \mid h_\pm\binom{n-2}{3}$ implies that $3^9 \mid n-1$. Now, the integrality of $\Phi_\pm(X_{K_5})$ implies that, modulo $3^{10}$,
\begin{multline*}
0 \equiv a^2 \binom{n}{5} \pm a \binom{n-2}{3} (n-1) \mp 2a(n-1)\equiv \binom{n}{5} - \binom{n-2}{3}(n-1) + 2(n-1) \\ \overset{\labelcref{eq:binom1}}{\equiv} (n-1)\paren*{-\frac{1}{20} + 1 + 2},
\end{multline*}
which implies that $3^{10} \mid n-1$ as $-\frac{1}{20} + 1 + 2 \not\equiv 0 \pmod{3}$. The only if direction follows from the fact that we have $3^9 \mid \binom{n}{v(H)}$, $3^{10} \mid \binom{n}{5}$, and $3^{10} \mid h_\pm$.

\subsubsection{Case 3: $n \equiv 2 \pmod 3$}
We aim to show that $n$ is $(p, \pm)$-hyperproportional if and only if $3^{10} \mid n-2$, for both choices of $p$ and both signs.

For the only if direction, note that since $3 \nmid h_\pm$, we must have $3^9 \mid \binom{n-2}{3}$, so $3^{10} \mid n-2$. For the if direction, observe that $\Phi_\pm(X_{K_2}) = \frac{1}{3}(a\binom{n}{2} + h_\pm)$ is an integer since $\binom{n}{2} \equiv 1 \pmod{3}$ and $h_\pm \equiv -a \pmod{3}$. Moreover, for $v(H) \geq 3$, we have that $3^9$ divides $\binom{n}{v(H)}$ and $\binom{n-2}{v(H)-2}$, so it suffices to check $H = K_5$. To do this, we need to show that
\[3^{10} \mid a^2 \binom{n}{5} + a\binom{n-2}{3} h_\pm,\]
which since $h_\pm \equiv -a\pmod{3}$ is equivalent to $\binom{n}{5} \equiv \binom{n-2}{3} \pmod{3^{10}}$. This follows from \labelcref{eq:binom2}.

\printbibliography
\end{document}